\documentclass[12pt]{amsart}
\usepackage{amsmath}
\usepackage{amsfonts}
\usepackage{amsthm}

\theoremstyle{plain} \numberwithin{equation}{section}
\newtheorem{Theorem}{Theorem}
\newtheorem{Lemma}[Theorem]{Lemma}
\newtheorem{Proposition}[Theorem]{Proposition}
\newtheorem{Corollary}[Theorem]{Corollary}

\newtheorem{Definition}[Theorem]{Definition}

\theoremstyle{remark}

\date{}

\title[Spectral Decompositions]
{Unconditional convergence of spectral decompositions of 1D Dirac
operators with regular boundary conditions}

\thanks{This paper was completed at the 
Mathematisches Forschungsinstitut Oberwolfach 
during our three week stay there 
in August 2010 within the Research in Pairs Programme.
We appreciate the hospitality and creative atmosphere of the Institute.} 
 
\author{Plamen Djakov}
\author{Boris Mityagin}
\begin{document}
\address{Sabanci University, Orhanli,
34956 Tuzla, Istanbul, Turkey}
 \email{djakov@sabanciuniv.edu}
\address{Department of Mathematics,
The Ohio State University,
 231 West 18th Ave,
Columbus, OH 43210, USA}
\email{mityagin.1@osu.edu}

\begin{abstract}
One dimensional  Dirac operators $$ L_{bc}(v) \,
y = i \begin{pmatrix} 1 & 0 \\ 0 & -1
\end{pmatrix}
\frac{dy}{dx}  + v(x) y, \quad y = \begin{pmatrix} y_1\\y_2
\end{pmatrix},  \quad
x\in[0,\pi],$$ considered
with  $L^2$-potentials $ v(x) = \begin{pmatrix} 0
& P(x) \\ Q(x) & 0 \end{pmatrix} $
and subject to regular boundary conditions ($bc$), have
discrete spectrum.

For  strictly regular $bc,$  it is shown that
every eigenvalue of the free operator $L^0_{bc}$
is simple and has the form
$\lambda_{k,\alpha}^0 = k + \tau_\alpha $  where $
\; \alpha \in \{1,2\},  \; k \in 2 \mathbb{Z} $
and $\tau_\alpha =\tau_\alpha (bc);$
if  $|k|>N(v, bc) $ each of the discs
$D_k^\alpha = \{z: \; |z-\lambda_{k,\alpha}^0| <\rho =\rho (bc) \} , $
$\alpha \in \{1,2\}, $  contains
exactly one simple eigenvalue $\lambda_{k,\alpha} $ of
$L_{bc} (v) $ and   
$(\lambda_{k,\alpha} -\lambda_{k,\alpha}^0)_{k\in 2\mathbb{Z}}  $
is an $\ell^2 $-sequence.
Moreover, it is proven that the  root projections $
 P_{n,\alpha} = \frac{1}{2\pi i} \int_{\partial D^\alpha_n}
(z-L_{bc} (v))^{-1} dz $ satisfy
the Bari--Markus condition $$\sum_{|n| > N}
\|P_{n,\alpha} - P_{n,\alpha}^0\|^2 < \infty, \quad n \in 2\mathbb{Z}, $$
where $P_n^0 $
are the root projections of the free operator $L^0_{bc}.$
Hence, for strictly regular $bc,$ there is a Riesz basis
consisting of root functions (all but finitely many being eigenfunctions).
Similar results are obtained for regular but not strictly regular $bc$  --
then in general
there is no Riesz basis consisting of root functions but we prove
that the corresponding system of  two-dimensional root projections
is a Riesz basis of projections.

\end{abstract}
\maketitle

{\it Keywords}: Dirac operators, Riesz bases, regular boundary conditions

 {\it MSC:} 47E05, 34L40, 34L10.\vspace{5mm}

{\em CONTENT}

\begin{enumerate}
\item[1.]{Introduction}

\item[2.]{Technical preliminaries; Riesz systems of projections}

\item[3.]{General regular and strictly regular boundary conditions}

\item[4.]{Matrix representation of $L_{bc}$ and its resolvent $R_{bc}(\lambda)$}

\item[5.]{Localization of spectra}

\item[6.]{Bari--Markus property in the case of strictly regular boundary conditions}

\item[7.]{Bari--Markus property in the case of regular but not
strictly regular boundary conditions}

\item[8.]{Miscellaneous; pointwise convergence and equiconvergence}

\end{enumerate}

\section{Introduction}

  Spectral theory of non-selfadjoint boundary value problems ($BVP$) for
ordinary  differential equations on a finite interval $I$ goes back
to the classical works of Birkhoff \cite{Bir1,Bir2} and Tamarkin
\cite{Tam1,Tam2,Tam3}. They introduced a concept of regular ($R$)
 boundary conditions ($bc$) and investigated asymptotic behavior
 of eigenvalues and eigenfunctions of such problems. Moreover, they
proved that the system of eigenfunctions and associated functions
($SEAF$) of a regular $BVP$ is complete.

   More subtle is the question whether $SEAF$ is a basis or an
unconditional basis in the Hilbert space $H^0 = L^2(I)$. N. Dunford
\cite{Du58} (see also \cite{DS71}),
V. P. Mikhailov \cite{Mi62}, G. M.  Keselman
\cite{Ke64} independently proved that the $SEAF$ is an
unconditional, or Riesz, basis if $bc$ are strictly regular ($SR$).
This property is lost if $bc$ are $R \setminus SR$, i.e., regular
but not strictly regular; unfortunately, this is just the case of
periodic ($Per^+$) and anti-periodic ($Per^-$) $bc.$  But A.  A.
Shkalikov \cite{Sh79, Sh82, Sh83} proved that in $R \setminus SR$
cases a proper chosen finite-dimensional projections form a Riesz
basis of projections.

Dirac operators
\begin{equation}
\label{i1}
Ly = i \begin{pmatrix}    1  & 0 \\  0 & -1
\end{pmatrix}
\frac{dY}{dx} + v(x) Y, \quad Y =
 \begin{pmatrix}    y_1 \\  y_2
\end{pmatrix}, \quad   v(x)=\begin{pmatrix}  0& P(x) \\ Q(x) & 0
\end{pmatrix}
\end{equation}
with $P, Q \in L^2 (I),$ and more general operators
\begin{equation}
\label{i2}
My = i B \frac{dY}{dx} + v(x) Y, \quad Y =(y_j (x))_1^d,
\end{equation}
where $B$ is a $d \times d $-matrix and $v(x) $ is
 a $d \times d$ matrix-valued $L^2 (I)$ function
 bring new difficulties. One of them comes from the fact that the values
of the resolvent  $(\lambda - L_{bc})^{-1}$ are not trace class operators.

  For general system (\ref{i2}) M. M. Malamud and L. L. Oridoroga  \cite{MalOr} gave
sufficient conditions for the completeness and minimality of the $SEAF$
in the case of  regular $BVP.$

The Riesz basis property for $2 \times 2 $  Dirac
operators (\ref{i1})  was proved by I. Trooshin and M. Yamamoto
\cite{TY01, TY02} in the case of separated $bc$  and $v \in L^2$.
S. Hassi and
L. L. Oridoroga \cite{HO09}
proved the Riesz basis property for (\ref{i2}) when $B =
\begin{pmatrix} a  & 0 \\  0  & -b   \end{pmatrix}, $
with $a,b >0,$  for separated  $bc$  and $v \in  C^1 (I).$

  B. Mityagin \cite{Mit03}, \cite[Theorem 8.8]{Mit04}  proved that {\em periodic (or
anti-periodic)} $bc$ give a rise of a Riesz system of 2D projections
(or 2D invariant subspaces) under the smoothness restriction $P, Q
\in H^\alpha, \; \alpha > 1/2,$ on the potentials $v$ in (\ref{i1}).
The authors removed that restriction in \cite{DM20}, where the same
result is obtained for {\em any} $L^2$ potential $v.$ This became possible
in the framework of the
 general approach
 to analysis of invariant (Riesz)
subspaces and their closeness to 2D subspaces of the free
operator developed and used by the authors in
\cite{DM3, DM5, DM6, DM7, DM15}.

Now we extend these results to Dirac operators with {\em any
regular} $bc,$
  which requires a careful analysis of
  {\em regular} and {\em strictly regular}
(a la Birkhoff-Tamarkin) $bc$
themselves -- Section 3 describes these $bc$ and
give explicit form of the $SEAF$ (Lemmas 5, 6, 7)
for SR and $R\setminus SR$ $bc$ in
the case
of the free Dirac operator. Section 2 reminds the elementary
geometry of Riesz bases
or Riesz systems of projections in a Hilbert space (see \cite{B,M,GK}).
 In Section 4 and 5 we
study  the analytic properties of the resolvent $R_{bc} (\lambda) =
( \lambda - L_{bc})^{-1}$ with $v \in L^2.$  $SR$ and $R\setminus
SR$ cases differ in some technical details, and Theorems 12 and 14
accordingly take care about localization of $L_{bc}$'s spectra. Now
(Sections 6, 7) the representation of projections as Cauchy--Riesz
integrals of the resolvent is used to get Bari--Markus property of
the Riesz system for $L_{bc}$. In the $SR$ case this leads (Theorem
15) to Riesz basis property of the $SEAF;$ in the $R\setminus SR$
case the system of 2D projections of root subspaces is a Riesz
system (Theorem 20).

\section{Technical preliminaries about Riesz systems of projections}

Here we recall some basic facts about Hilbert--Schmidt operators,
Riesz bases, etc. All Hilbert spaces that we consider are supposed
to be separable. \vspace{2mm}

1. {\em Hilbert--Schmidt operators.}

Let $H$ be a Hilbert space. A linear operator $T: H \to H$ is {\em
Hilbert--Schmidt operator} if its {\em Hilbert--Schmidt norm}
$\|T\|_{HS}$ is finite, where
\begin{equation}
\label{p1} \|T\|^2_{HS}: = \sum_{\gamma \in \Gamma} \|Te_\gamma \|^2
=\sum_{\gamma, \beta \in \Gamma} |\langle Te_\gamma,e_\beta \rangle
|^2
\end{equation}
with $(e_\gamma, \gamma \in \Gamma)$ being any orthonormal basis
(o.n.b.) in $H.$ The following lemma summarizes some of the
properties of Hilbert--Schmidt operators and Hilbert--Schmidt norm.

\begin{Lemma}
\label{lem1} Let $T: H \to H$ and $S: H \to H$ be linear
operators.

 (a) $\|T\|_{HS}$ in (\ref{p1}) does not depend on the choice of
o.n.b. $(e_\gamma).$

(b) $\|T\|_{HS}$ is a norm such that $\|T^*\|_{HS}=\|T\|_{HS}.$

(c) $\|T\| \leq \|T\|_{HS}$

(d) If $T$ is Hilbert--Schmidt and $S$ is bounded, then $ S T $ and
$ T S $ are Hilbert--Schmidt operators, and $$
\|ST\|_{HS},\|TS\|_{HS} \leq\|T\|_{HS} \cdot \|S\|$$

(e) Every Hilbert--Schmidt operator is compact.
\end{Lemma}
We refer to \cite{GK,RS} for proofs of these properties and more
details about Hilbert--Schmidt operators.\vspace{2mm}

2. {\em Riesz bases.}

Let $H$ be a Hilbert space, and let $(e_\gamma, \, \gamma \in
\Gamma)$ be an o.n.b. in $H.$ If $A: H\to H$ is an automorphism,
then the system
\begin{equation}
\label{p2} f_\gamma = A e_\gamma, \quad  \gamma \in \Gamma,
\end{equation}
is an unconditional basis in $H.$ Indeed, for each $x\in H$ we have
$$ x= A(A^{-1} x )= A \left (\sum_\gamma \langle A^{-1} x,e_\gamma
\rangle e_\gamma \right) $$
$$= \sum_\gamma \langle x,(A^{-1})^*e_\gamma
\rangle f_\gamma =\sum_\gamma \langle x,\tilde{f}_\gamma \rangle
f_\gamma,$$ so $(f_\gamma)$ is a basis, and its biorthogonal
system is
\begin{equation}
\label{p2a} \tilde{f}_\gamma =(A^{-1})^* e_\gamma, \quad \gamma \in
\Gamma.
\end{equation}
Moreover, it follows that
\begin{equation} \label{p3} 0< c \leq
\|f_\gamma\| \leq  C, \quad m^2\|x\|^2 \leq \sum_\gamma |\langle
x,\tilde{f}_\gamma \rangle|^2 \|f_\gamma\|^2 \leq M^2 \|x\|^2,
\end{equation}
with $ c= 1/\|A^{-1}\|, \; C=\|A\|,  \; M= \|A\|\cdot \|A^{-1}\|$
and $m=1/M.$

A basis of the form (\ref{p2}) is called {\em Riesz basis.} One can
easily see that the property (\ref{p3}) characterizes Riesz bases,
i.e., a basis $(f_\gamma)$ is a Riesz bases if and only if
(\ref{p3}) holds with some constants $C\geq c>0$ and $M \geq m >0.$
Another characterization of Riesz bases gives the following
assertion (see \cite[Chapter 6, Section 5.3, Theorem 5.2]{GK}):  {\em If
$(f_\gamma)$ is a normalized  basis (i.e., $\|f_\gamma\|=1 \;
\forall \gamma $), then it is a Riesz basis if and only if it is
unconditional.}

Let $(f_\gamma)$ be a fixed Riesz basis in $H$ of the form
(\ref{p2}). For each Hilbert--Schmidt operator $T$  we consider
\begin{equation}
\label{p4}  \|T\|^*_{HS} = \left (\sum_{\gamma, \beta} |\langle
Tf_\gamma, \tilde{f}_\beta \rangle|^2 \right )^{1/2}.
\end{equation}

Then  $\|T\|^*_{HS}$ is a norm which is equivalent to
$\|T\|_{HS}.$ Indeed, in view of (\ref{p1})--(\ref{p2a}),
$$(\|T\|^*_{HS})^2=\sum_{\gamma, \beta} |\langle TA e_\gamma, (A^{-1})^*
e_\beta \rangle|^2 $$
$$=\sum_{\gamma, \beta} |\langle A^{-1}TA e_\gamma,
e_\beta \rangle|^2 =\|A^{-1}TA\|^2_{HS}.$$ Therefore, in view of
Lemma \ref{lem1},
$$ \|T\|^*_{HS}=\|A^{-1}TA\|_{HS}\leq  M \|T\|_{HS}$$ with $M=
\|A\|\cdot \|A^{-1}\|.$ On the other hand, by the same argument,
$$ \|T\|_{HS} =\|A(A^{-1}TA)A^{-1}\|_{HS} \leq M
\|A^{-1}TA\|_{HS}=M\|T\|^*_{HS}. $$ \vspace{2mm}

3. {\em Riesz bases of projections and Bari--Markus Theorem}

Let $H$ be a Hilbert space. A family of bounded finite--dimensional
projections $\{P_\gamma: H\to H , \, \gamma \in \Gamma\}$ is called
{\em unconditional basis of projections} if the following conditions
hold:
\begin{align}
P_\alpha P_\beta =0  \quad \text{if} \;\; \alpha \neq \beta,
\quad P_\alpha^2 = P_\alpha; \label{p5}\\
x= \sum_{\gamma \in \Gamma} P_\gamma (x) \quad \forall x \in H,\label{p6}
\end{align}
where the series converge unconditionally.

Obviously, if $(f_\gamma) $ is an unconditionsl basis in $H$ then the system of
one--dimensional projections $P_\gamma (x) = \tilde{f}_\gamma (x)
f_\gamma$ is a basis of projections in $H, $ and vice versa, every
basis of one dimensional projections can be obtained in that way
from some basis.

If $(Q_\gamma)$ is a basis of orthogonal projections (i.e.,
$Q_\gamma^* = Q_\gamma $),  the Pythagorian theorem implies  $
\sum_\gamma \|Q_\gamma x\|^2 = \|x\|^2. $

 We say
that the family of projections $(P^0_\gamma, \, \gamma\in \Gamma)$
is a {\em Riesz basis of projections} if
\begin{equation}
\label{p10} P^0_\gamma = A Q_\gamma A^{-1}, \quad \gamma\in \Gamma,
\end{equation}
where $A:H\to H $ is an isomorphism and $(Q_\gamma, \, \gamma\in
\Gamma)$ is a basis of orthogonal projections.

If (\ref{p10}) holds, then
\begin{equation}
\label{p11} \sum_\gamma \|P^0_\gamma x \|^2  \leq \|A\|^2
\sum_\gamma \|Q_\gamma A^{-1}x\|^2 = \|A\|^2 \| A^{-1} x\|^2  \leq M^2
\|x\|^2
\end{equation}
with $M= \|A\| \|A^{-1}\|.$

The following statement is a version of the Bari-Markus theorem (see
\cite{GK}, Ch.6, Sect. 5.3, Theorem 5.2).

\begin{Theorem}
\label{thm2} Suppose that $(P_\gamma,\, \gamma\in \Gamma)$  is a
family of bounded finite dimensional projections in a Hilbert space
$H$ such that
\begin{equation}
\label{p14} P_\alpha P_\beta =0  \quad \text{if} \;\; \alpha \neq
\beta.
\end{equation}
If there is a Riesz basis of projections $(P_\gamma^0,\, \gamma\in
\Gamma)$ such that
\begin{equation}
\label{p15}  \dim P_\gamma = \dim P_\gamma^0, \quad   \gamma \in
\Gamma,
\end{equation}
and
\begin{equation}
\label{p16} \sum_{\gamma\in \Gamma} \| P_\gamma -P^0_\gamma\|^2 <
\infty,
\end{equation}
then $(P_\gamma)$ is a Riesz basis of projections in $H.$
\end{Theorem}

\begin{proof} Let the projections $P_\gamma^0$
be given by (\ref{p10}). In view of (\ref{p16}), there is a finite
subset $\Gamma_1 \subset \Gamma $ such that

\begin{equation}
\label{p17} \sum_{\Gamma \setminus \Gamma_1} \| P_\gamma
-P_\gamma^0\|^2 M^2 < \frac{1}{4},
\end{equation}
where the constant $M= \|A^{-1}\|\|A\|$ comes from (\ref{p11}).
Consider the operators
$$
T x= \sum_{\Gamma\setminus \Gamma_1} (P_\gamma -P^0_\gamma)P^0_\gamma
x, \quad Bx = \sum_{\Gamma_1} P^0_\gamma x + \sum_{\Gamma\setminus
\Gamma_1} P_\gamma P^0_\gamma x = x+ Tx.
$$
In view of (\ref{p11}) and (\ref{p17}), the Cauchy inequality yields
$$
\|Tx\|^2 \leq \left ( \sum_{\Gamma\setminus \Gamma_1} \|P_\gamma
-P^0_\gamma\| \|P^0_\gamma x \| \right )^2 \leq \sum_{\Gamma\setminus
\Gamma_1} \|P_\gamma -P^0_\gamma\|^2 \sum_{\Gamma\setminus \Gamma_1}
\|P^0_\gamma x\|^2 \leq \frac{1}{2} \|x\|^2.
$$
Therefore $ \|T\|< 1/2, $ which implies that $B:H\to H$ is an
isomorphism.

By the construction of the operator $B,$ if ${\alpha \in
\Gamma\setminus \Gamma_1}$ then  $B$ coincides on the subspace
$P_\alpha^0 (H)$ with the projection $P_\alpha,$ i.e.,
$$
BP_\alpha^0 x = P_\alpha P_\alpha^0 x \quad \text{for} \quad \alpha
\in \Gamma\setminus \Gamma_1, \;\; x\in H.
$$
Since  $\dim P_\alpha^0 (H)= \dim P_\alpha (H) <\infty,$  it follows that $B$
maps $ P_\alpha^0 (H)$ onto $ P_\alpha (H),$  which yields  $P_\alpha
= BP_\alpha^0 B^{-1}$ for $\alpha \in \Gamma\setminus \Gamma_1. $

Let $H^0_1, \, H^0_2, \, H_1, \, H_2 $ be, respectively,  the closed
linear spans of  $$\bigcup_{\Gamma_1} P_\gamma^0 (H), \quad
\bigcup_{\Gamma\setminus \Gamma_1} P_\gamma^0 (H), \quad
\bigcup_{\Gamma_1} P_\gamma (H), \quad \bigcup_{\Gamma\setminus
\Gamma_1} P_\gamma (H);
$$
then $H=H^0_1 \oplus H^0_2 = H_1 \oplus H_2,$ and
$B(H^0_2)=H_2.$ Since $\dim P_\gamma^0 (H)= \dim P_\gamma (H), $
there exists an isomorphism
$\tilde{B} : H \to H $ such that $\tilde{B} = B$  on $H^0_2 $ and
$\tilde{B} $ maps $P_\gamma^0 (H)$ onto $P_\gamma (H)$ for every $
\gamma.$ Thus, by (\ref{p10}) we obtain
$$
P_\gamma = \tilde{B} P_\gamma^0 \tilde{B}^{-1} =\tilde{B} AQ_\gamma
A^{-1} \tilde{B}^{-1}=(\tilde{B} A)Q_\gamma  (\tilde{B}A)^{-1}, \quad
\gamma \in \Gamma,
$$
which proves that $(P_\gamma ) $ is a Riesz basis of projections.
\end{proof}

\section{General regular and strictly regular boundary conditions}

We consider the Dirac operators $L=L(v)$ given by
(\ref{i1}) on the interval $I=[0,\pi] $ and set $L^0 = L(0).$  
In the following, the Hilbert space $L^2 (I, \mathbb{C}^2) $ is
regarded equipped with the scalar product
\begin{equation}
\label{0} \left \langle \begin{pmatrix} f_1 \\ f_2
\end{pmatrix},\begin{pmatrix} g_1 \\ g_2  \end{pmatrix}  \right
\rangle =\frac{1}{\pi} \int_0^\pi \left (f_1 (x) \overline{g_1 (x)}
+ f_2 (x) \overline{g_2 (x)} \right ) dx.
\end{equation}

1. A general boundary condition for the operator $L^0 $ (or $L$)
is given by a system of two linear equations
\begin{eqnarray}
\label{1} a_1 y_1 (0) +b_1 y_1 (\pi) + a_2 y_2 (0) + b_2 y_2
(\pi)=0  \\ \nonumber c_1 y_1 (0) +d_1 y_1 (\pi) + c_2 y_2 (0) +
d_2 y_2 (\pi)=0
\end{eqnarray}
Let  $A_{ij}$  denote the  $2\times 2$ matrix
 formed by the $i$-th and $j$-th columns of the
matrix
\begin{equation}
\label{2} \left [
\begin{array}{cccc}
a_1 & b_1 & a_2 & b_2\\ c_1 & d_1& c_2 & d_2
\end{array}
\right ],
\end{equation}
and let $|A_{ij}|$ denote the determinant of the matrix $A_{ij}.$
Each solution of the equation
\begin{equation}
\label{2a} L^0 y =\lambda y,\quad y=\begin{pmatrix} y_1\\y_2
\end{pmatrix}
\end{equation}
has the form
\begin{equation}
\label{3} y =\begin{pmatrix} \xi e^{-i\lambda x}\\ \eta
e^{i\lambda x}
\end{pmatrix}.
\end{equation}
It satisfies the boundary condition (\ref{1}) if and only if $(\xi,
\eta)$ is a solution of the system of two linear equations
\begin{eqnarray}
\label{4} \xi (a_1 + b_1 z^{-1} ) + \eta (a_2 + b_2 z) =0  \\
\nonumber \xi (c_1 + d_1 z^{-1} ) + \eta (c_2 + d_2 z) =0
\end{eqnarray}
where $z= \exp (i\pi\lambda).$  Therefore, we have a non-zero
solution $y$
 if and only if the determinant of (\ref{4}) is
zero, which is equivalent to the quadratic equation
\begin{equation}
\label{5} |A_{14}| z^2 + (|A_{13}| + |A_{24}|) z + |A_{23}|=0.
\end{equation}
\begin{Definition}
The boundary condition (\ref{1}) is called: {\bf regular} if
\begin{equation}
\label{6} |A_{14}| \neq 0, \quad |A_{23}| \neq 0,
\end{equation}
and {\bf strictly regular} if additionally
\begin{equation}
\label{7} (|A_{13}|+|A_{24}|)^2 \neq 4|A_{14}| |A_{23}|
\end{equation}
holds.
\end{Definition}

Of course, (\ref{7}) is equivalent to saying that the quadratic
equation (\ref{5}) has two {\em distinct} roots.

From now on we consider only regular boundary conditions. We multiply
from the left the system (\ref{1}) and the $2\times 4$ matrix
(\ref{2}) by the matrix $A^{-1}_{14}. $ This gives us an equivalent
to (\ref{1}) system
\begin{eqnarray}
\label{8a}  y_1 (0) +b y_1 (\pi) + a y_2 (0) =0
\\ \nonumber d y_1 (\pi) + c y_2 (0) +  y_2 (\pi)=0
\end{eqnarray}
which matrix has the form
\begin{equation}
\label{8} \left [
\begin{array}{cccc}
1 & b & a & 0\\ 0 & d& c & 1
\end{array}
\right ],
\end{equation}
and $\left [ \begin{array}{cc}  b & a \\  d& c
\end{array}
\right ] = A_{14}^{-1}A_{23}. $

In the following we consider only boundary conditions in the form
(\ref{8a}) with matrices (\ref{8}). Then
\begin{equation}
\label{9} |A_{14}|=1, \quad |A_{13}|=c, \quad |A_{24}|=b,\quad
|A_{23}| = bc-ad.
\end{equation}
Condition (\ref{6}) means that
\begin{equation}
\label{10}  |A_{23}| = bc-ad \neq 0,
\end{equation}
and (\ref{7}) takes the form
\begin{equation}
\label{11} (b-c)^2 +4ad \neq 0.
\end{equation}

Now the system (\ref{4}) becomes
\begin{eqnarray}
\label{12} \xi (1 + b z^{-1} ) + \eta a=0  \\ \nonumber \xi d
z^{-1}  + \eta (c +  z) =0
\end{eqnarray}
and the equation (\ref{5}) becomes
\begin{equation}
\label{13} z^2 + (b+c)z + bc-ad =0.
\end{equation}

Notice that (\ref{12}) means
\begin{equation}
\label{19} \left [
\begin{array}{cc}
1+b/z &a\\ d/z &  c+z \end{array} \right ]
\begin{pmatrix} \xi \\ \eta\end{pmatrix} =  \left [
\begin{array}{cc}
z+b &a\\ d &  c+z \end{array} \right ]
\begin{pmatrix} \xi/z \\ \eta\end{pmatrix}
=0.
\end{equation}
From here (or, since the change of variable $ z \to -w $ transforms
(\ref{13}) into the characteristic equation of the matrix $A_{23}$)
we get the following.

\begin{Lemma}
\label{lem2} The number $z$ is a root of (\ref{13}) if and only if
$-z$ is an eigenvalue of the matrix $A_{23} =\left [
\begin{array}{cc} b &a\\ d &  c \end{array} \right ].$
Moreover, $\begin{pmatrix} \xi \\ \eta\end{pmatrix}$ is a non-zero
solution of (\ref{12}) if and only if $-z$ is an eigenvalue
of the matrix $A_{23}$ and $\begin{pmatrix} \xi/z \\
\eta\end{pmatrix}$ is an eigenvector of $A_{23} $ corresponding to
$-z.$
\end{Lemma}

2. {\em Strictly Regular boundary conditions.} In this case the
conditions (\ref{10}) and (\ref{11}) guarantee that quadratic
equation (\ref{13}) has two distinct nonzero roots $z_1$ and $z_2,$
so the matrix $A_{23}$ has two {\em distinct} eigenvalues $-z_1,
-z_2.$ Let us fix a pair of corresponding eigenvectors
$\begin{pmatrix} \alpha_1 \\ \alpha_2\end{pmatrix}$ and $
\begin{pmatrix} \beta_1 \\ \beta_2 \end{pmatrix}. $ Then
\begin{equation}
\label{20}
\begin{pmatrix} \alpha_1 \\ \alpha_2\end{pmatrix}  \quad and \quad
\begin{pmatrix} \beta_1 \\ \beta_2 \end{pmatrix} \quad are \: linearly \;
independent
\end{equation}
because $z_1 \neq z_2. $ Therefore, the matrix $\begin{pmatrix}
\alpha_1 & \beta_1\\ \alpha_2 & \beta_2
\end{pmatrix}$ is invertible; we set
\begin{equation}
\label{20a}
\begin{pmatrix} \alpha^\prime_1 & \alpha^\prime_2\\
\beta^\prime_1 & \beta^\prime_2 \end{pmatrix} := \begin{pmatrix} \alpha_1 & \beta_1\\
\alpha_2 & \beta_2 \end{pmatrix}^{-1}.
\end{equation}
Moreover, in view of Lemma \ref{lem2}, the vectors $\begin{pmatrix}
\alpha_1 z_1
\\ \alpha_2 \end{pmatrix}$ and $
\begin{pmatrix} \beta_1z_2 \\ \beta_2 \end{pmatrix} $
are solutions of the system (\ref{12}).

Let $\tau_1$ and $ \tau_2 $ be chosen so that
\begin{equation}
\label{14} z_1 = e^{i\pi \tau_1}, \quad z_2 = e^{i\pi \tau_2}
\end{equation}
and
\begin{equation}
\label{15}  |Re\, \tau_1 - Re\, \tau_2| \leq 1, \quad |Re \, \tau_1| \leq 1.
\end{equation}
Then we have
\begin{equation}
\label{17} z_1 = e^{i\pi \lambda} \quad  \Leftrightarrow \quad
\lambda = \tau_1 + k, \; k \in 2\mathbb{Z}
\end{equation}
and
\begin{equation}
\label{18} z_2 = e^{i\pi \lambda} \quad  \Leftrightarrow \quad
\lambda = \tau_2 + m, \; m \in 2\mathbb{Z}.
\end{equation}
The right--hand sides of (\ref{17}) and (\ref{18}) give all
eigenvalues of $L^0.$
 For each $\lambda $ in the two infinite series given
by (\ref{17}) and (\ref{18}) we have an eigenvector of $L^0$ of the
form (\ref{3}) with $$\begin{pmatrix} \xi
\\ \eta \end{pmatrix}=  \begin{pmatrix} \alpha_1z_1
\\ \alpha_2 \end{pmatrix} \quad \text{if} \;\lambda = \tau_1 +m, \quad
\begin{pmatrix} \xi
\\ \eta \end{pmatrix}=  \begin{pmatrix} \beta_1 z_2
\\ \beta_2 \end{pmatrix} \quad \text{if} \; \lambda = \tau_2 +m. $$

Thus, the operator $L^0$ subject to the boundary conditions
(\ref{2a}) with matrix (\ref{8}), has the following two series of
eigenvectors:
\begin{equation}
\label{21}\Phi^1 =\{\varphi^1_k, \; k \in 2\mathbb{Z} \}, \qquad
\varphi^1_k :=\begin{pmatrix} z_1 \alpha_1 e^{-i (\tau_1+k)x}
\\  \alpha_2 e^{i (\tau_1+k)x} \end{pmatrix}=
\begin{pmatrix}  \alpha_1 e^{i \tau_1 (\pi-x)} e^{-ikx}
\\  \alpha_2 e^{i \tau_1 x} e^{ikx}\end{pmatrix}
\end{equation}
and
\begin{equation}
\label{22}\Phi^2 =\{\varphi^2_m, \; m \in 2\mathbb{Z} \}, \qquad
\varphi^2_m :=\begin{pmatrix} z_2 \beta_1 e^{-i (\tau_2+m)x}
\\  \beta_2 e^{i (\tau_2+m)x} \end{pmatrix}=
\begin{pmatrix}  \beta_1 e^{i \tau_2 (\pi-x)} e^{-imx}
\\  \beta_2 e^{i \tau_2 x} e^{imx}\end{pmatrix}
\end{equation}

\begin{Lemma}
\label{lem3} The system $\Phi =\Phi^1 \cup \Phi^2$ is a Riesz basis
in the space  $L^2 (I, \mathbb{C}^2), \; I= [0,\pi].$ Its
biorthogonal system is $\tilde{\Phi} =\tilde{\Phi}^{1} \cup
\tilde{\Phi}^{2},$ where
\begin{equation}
\label{21*}\tilde{\Phi}^{1} =\{\tilde{\varphi}^1_k, \; k \in
2\mathbb{Z} \}, \qquad \tilde{\varphi}^1_k :=
\begin{pmatrix}  \overline{\alpha^\prime_1} e^{i \overline{\tau_1} (\pi-x)} e^{-ikx}
\\  \overline{\alpha^\prime_2} e^{i \overline{\tau_1} x} e^{ikx}\end{pmatrix},
\end{equation}
and
\begin{equation}
\label{22*}\tilde{\Phi}^{2} =\{\tilde{\varphi}^2_m, \; m \in
2\mathbb{Z} \}, \qquad \tilde{\varphi}^2_m :=
\begin{pmatrix}  \overline{\beta^\prime_1} e^{i \overline{\tau_2} (\pi-x)} e^{-imx}
\\ \overline{\beta^\prime_2} e^{i \overline{\tau_2} x} e^{imx}\end{pmatrix},
\end{equation}
with $\alpha_1^\prime, \alpha_2^\prime, \beta_1^\prime, \beta_2^\prime $   coming from (\ref{20a}).
\end{Lemma}

\begin{proof}
The system $E=E^1 \cup E^2, $ where
\begin{equation}
\label{23} E^\nu =\{e^\nu_k, \; k \in 2\mathbb{Z} \}; 
\;\; \nu = 1,2; 
\qquad e^1_k :=\begin{pmatrix} e^{i kx}\\
0
\end{pmatrix}, \quad  e^2_m :=\begin{pmatrix} 0
\\   e^{i mx} \end{pmatrix},
\end{equation}
is an orthonormal basis in $L^2 (I, \mathbb{C}^2).$

Consider the operator $A: L^2 (I, \mathbb{C}^2) \to L^2 (I,
\mathbb{C}^2)$ defined by
\begin{equation}
\label{24} A
\begin{pmatrix}  f\\ g \end{pmatrix} =
\begin{pmatrix}  \alpha_1 e^{i \tau_1 (\pi-x)} f(\pi -x)
\\  \alpha_2 e^{i \tau_1 x} f(x) \end{pmatrix}
+ \begin{pmatrix}  \beta_1 e^{i \tau_2 (\pi-x)} g(\pi -x)
\\  \beta_2 e^{i \tau_2 x} g(x) \end{pmatrix}.
\end{equation}
Since we have $\Phi = A(E),$ the lemma will be proved if we show
that $A$ is an isomorphism. Since the functions $\displaystyle e^{i
\tau_\nu x} $ and $\displaystyle e^{i \tau_\nu (\pi -x)}, \; \nu
=1,2, $ are bounded, it follows that $A$ is bounded operator. Let us
find its inverse. By (\ref{24}), the equation $$A
\begin{pmatrix}  f\\ g \end{pmatrix} = \begin{pmatrix}  F\\ G
\end{pmatrix}$$ is equivalent to the following system of two
linear equations in two unknowns $f,g:$
\begin{eqnarray}
\label{25} \alpha_1 e^{i \tau_1 x} f(x) +\beta_1  e^{i \tau_2 x} g(x)
= F(\pi -x),\\ \nonumber  \alpha_2 e^{i \tau_1 x} f(x) +\beta_2  e^{i
\tau_2 x} g(x) = G(x).
\end{eqnarray}
By (\ref{20}) and (\ref{20a}), we get $$
\begin{pmatrix}  e^{i \tau_1 x} f(x)
\\  e^{i \tau_2 x} g(x) \end{pmatrix} =
\begin{pmatrix}  \alpha^\prime_1 F (\pi-x) + \alpha^\prime_2 G(x)
\\  \beta^\prime_1 F (\pi-x) + \beta^\prime_2 G(x) \end{pmatrix},$$
which leads to
\begin{equation}
\label{26} A^{-1}
\begin{pmatrix}  F\\ G \end{pmatrix} =
\begin{pmatrix} e^{-i \tau_1 x}[ \alpha^\prime_1 F (\pi-x) + \alpha^\prime_2
G(x)]
\\  e^{-i \tau_2 x} [\beta^\prime_1 F (\pi-x) + \beta^\prime_2 G(x)] \end{pmatrix},
\end{equation}
Now it is easy to see that $A^{-1} $ is bounded.

Let us find the adjoint operator of $A^{-1} $. Since $$ \left
\langle A^{-1} \begin{pmatrix} F\\G
\end{pmatrix},\begin{pmatrix} f\\0 \end{pmatrix}  \right \rangle
= \frac{1}{\pi} \int_0^\pi  \{ [ \alpha^\prime_1 F (\pi-x) +
\alpha^\prime_2 G(x)]e^{-i \tau_1 x}\overline{f(x)} dx $$ $$=
\frac{1}{\pi} \int_0^\pi \left ( F (x) \overline{
\overline{\alpha^\prime_1} f(\pi-x)e^{i \overline{\tau_1}(\pi -x)}} +
G(x) \overline{\overline{\alpha^\prime_2} f(x)e^{i
\overline{\tau_1}x} } \right ) dx, $$ we get $$
(A^{-1})^*\begin{pmatrix} f\\0
\end{pmatrix}= \begin{pmatrix}  \overline{\alpha^\prime_1}
f(\pi-x)e^{i \overline{\tau_1}(\pi -x)}\\ \overline{\alpha^\prime_2}
f(x)e^{i \overline{\tau_1}x} \end{pmatrix}.  $$ In an analogous way
it follows that
$$ (A^{-1})^*\begin{pmatrix} 0\\g
\end{pmatrix}= \begin{pmatrix}  \overline{\beta^\prime_1}
g(\pi-x)e^{i \overline{\tau_2}(\pi -x)}\\ \overline{\beta^\prime_2}
g(x)e^{i \overline{\tau_2}x} \end{pmatrix}.  $$ Thus,
\begin{equation}
\label{27} (A^{-1})^*
\begin{pmatrix}  f\\ g \end{pmatrix} =
\begin{pmatrix}  \overline{\alpha^\prime_1}
f(\pi-x)e^{i \overline{\tau_1}(\pi -x)}\\ \overline{\alpha^\prime_2}
f(x)e^{i \overline{\tau_1}x} \end{pmatrix} +
\begin{pmatrix}  \overline{\beta^\prime_1}
g(\pi-x)e^{i \overline{\tau_2}(\pi -x)}\\ \overline{\beta^\prime_2}
g(x)e^{i \overline{\tau_2}x} \end{pmatrix}.
\end{equation}
By (\ref{p2a}), the system $\Phi$ is a Riesz basis, and its
biorthogonal system is given by (\ref{21*}) and (\ref{22*}). This
completes the proof.
\end{proof}

The system $\tilde{\Phi} $ has the same form as $\Phi,$ so it is a
system of eigenvectors of $L^0$ subject to appropriate boundary
conditions. Indeed, let $S$ denote the matrix $\begin{pmatrix} b
&a \\ d  &c \end{pmatrix},$ and let
\begin{equation}
\label{29} (S^{-1})^* =
\begin{pmatrix} \tilde{b} &\tilde{a}\\ \tilde{d}  &\tilde{c}
\end{pmatrix}.
\end{equation}
If $\begin{pmatrix} \alpha_1\\ \alpha_2 \end{pmatrix}$ and
$\begin{pmatrix} \beta_1 \\ \beta_2 \end{pmatrix}$ are eigenvectors
of $S$  corresponding to $-z_1 $ and $- z_2 $ as in (\ref{20}) and
Lemma \ref{lem2}, then
\begin{equation}
\label{30} S \begin{pmatrix} \alpha_1&\beta_1\\ \alpha_2 & \beta_2
\end{pmatrix}=
\begin{pmatrix} -z_1 \alpha_1 & -z_2 \beta_1\\ -z_1\alpha_2& -z_2 \beta_2 \end{pmatrix}.
\end{equation}
Let us mention that  the relation (\ref{30})  determines the
matrix $S$ if numbers $z_1 \neq z_2 $  and $\alpha_1, \alpha_2, \beta_1, \beta_2$
 satisfying (\ref{20})   are given.

 In view of (\ref{21}), taking the inverse matrices of both
sides of (\ref{30}), and then passing to adjoint matrices, we get
$$ (S^{-1})^*
\begin{pmatrix} \overline{\alpha^\prime_1}&  \overline{\beta^\prime_1}\vspace{1mm}\\
\overline{\alpha^\prime_2} &  \overline{\beta^\prime_2}
\end{pmatrix}=\begin{pmatrix} -\frac{1}{\overline{z_1}}\overline{\alpha^\prime_1}&
-\frac{1}{\overline{z_2}}\overline{\beta^\prime_1} \vspace{1mm}\\
-\frac{1}{\overline{z_1}}\overline{\alpha^\prime_2} &
-\frac{1}{\overline{z_2}}\overline{\beta^\prime_2} \end{pmatrix}. $$
This means that $-1/\overline{z_1},-1/\overline{z_2}$ are the
eigenvalues of the matrix $(S^{-1})^*,$ and $\begin{pmatrix}
\overline{\alpha^\prime_1} \vspace{1mm}\\ \overline{\alpha^\prime_2}
\end{pmatrix},$
$\begin{pmatrix} \overline{\beta^\prime_1} \vspace{1mm}\\
\overline{\beta^\prime_2}
\end{pmatrix}$ is a pair of corresponding (linearly independent)
eigenvectors.

Consider the boundary conditions that correspond to the matrix
\begin{equation}
\label{31}
\begin{pmatrix} 1 & \tilde{b} &\tilde{a}&0 \\ 0& \tilde{d}
&\tilde{c}&1
\end{pmatrix}
\end{equation}
with $\tilde{b}, \tilde{a}, \tilde{d},\tilde{c}$ coming from
(\ref{29}). In view of Lemma \ref{lem2},
$1/\overline{z_1}$ and $1/\overline{z_2}$ are the roots of the
characteristic equation (\ref{12}). But if $z=e^{i\tau \pi} $ then
$1/\overline{z}= \overline{e^{-i\tau \pi}}=e^{i\overline{\tau}
\pi}.$ Now, by (\ref{21}) --
(\ref{22*}), it follows that  $\tilde{\Phi}$ is a system of
eigenvectors of $L^0$ subject to the boundary conditions
(\ref{31}).

Next we show that, as usual, the biorthogonal system $\tilde{\Phi}
$ is the system of eigenvectors of the adjoint operator
$(L^0_{bc})^*$ (or, which is the same, of $L^0$ subject to adjoint
boundary conditions $bc^*$).

\begin{Lemma}
\label{lem4} Let $L_{bc}^0$ be a closed operator with boundary
conditions $bc$ defined by (\ref{8a}) and (\ref{8}). Then its
adjoint operator $(L_{bc})^*$ is  $L^0_{bc^*},$ where the boundary
conditions  $bc^*$ are given by the matrix (\ref{31}).
\end{Lemma}

\begin{proof}
With $f= \begin{pmatrix} f_1\\ f_2 \end{pmatrix}$ and $
g=\begin{pmatrix} g_1 \\ g_2  \end{pmatrix}$ such that $f,g  \in
H^1 (I,\mathbb{C}^2) $ we have

$$ \langle L^0 f,g\rangle - \langle f, L^0 g \rangle =
\frac{i}{\pi} \int_0^\pi \frac{d}{dx}\left (f_1 (x) \overline{g_1
(x)}-f_2 (x) \overline{g_2 (x)}\right  ) dx $$ $$ =\frac{i}{\pi}
\left ( f_1 (\pi) \overline{g_1 (\pi)} -f_1 (0) \overline{g_1 (0)}
-f_2 (\pi) \overline{g_2 (\pi)} +f_2 (0) \overline{g_2 (0)}\right
)$$ $$=\frac{i}{\pi} \left ( f_1 (\pi) \overline{g_1 (\pi)} + (b
f_1 (\pi) + a f_2 (0))\overline{g_1 (0)} + (df_1 (\pi) + cf_2
(0))\overline{g_2 (\pi)} +f_2 (0) \overline{g_2 (0)} \right ) $$
$$ =\frac{i}{\pi} \left ( f_1 (\pi) \overline{[\overline{b} g_1
(0) +g_1 (\pi) + \overline{d} g_2 (\pi)]} +f_2 (0)
\overline{[\overline{a} g_1 (0) + g_2 (0) +\overline{c}g_2 (\pi)]}
\right ).$$ In view of (\ref{8a}), one can easily see that $f_1
(\pi)$ and $f_2 (0)$ could be any numbers.

Therefore, the boundary conditions of the adjoint operator are
determined by the matrix
\begin{equation}
\label{33}
\begin{pmatrix} \overline{b} & 1 &0 &\overline{d} \\ \overline{a}&
0 &1&\overline{c}
\end{pmatrix}.
\end{equation}
In view of (\ref{29}), if we bring it to the equivalent form
(\ref{8})  multiplying from the left by $\begin{pmatrix}
\overline{b} &\overline{d} \\ \overline{a}&\overline{c}
\end{pmatrix}^{-1}, $ the result will be just (\ref{31}).
\end{proof}

3. {\em Dirichlet--type boundary conditions.} In general, for strictly
regular $bc,$ the spectrum of the operator $ L^0_{bc}$ consists of two
arithmetic progressions (\ref{17}) and (\ref{18}), with difference =
2. If
\begin{equation}
\label{62} b+c =0,
\end{equation}
then the equation (\ref{13}) has the following two roots
\begin{equation}
\label{63} z_1 =\sqrt{ad-bc}, \quad z_2 = -z_1.
\end{equation}
In view of (\ref{14}), in this case we have $ z_1 = e^{i\tau_1 \pi}$
and $ z_2 = e^{i\tau_2 \pi}$ with $\tau_2 = \tau_1 \pm 1.$
Therefore, the union of the
corresponding two arithmetic progressions (\ref{17}) and (\ref{18})
gives the spectrum of $L^0_{bc}$ in the form of one arithmetic
progression with difference 1:
\begin{equation}
\label{64} \lambda = \tau_1 + m, \quad m \in \mathbb{Z}.
\end{equation}
We call boundary conditions with the property (\ref{62}) {\em
Dirichlet-type boundary conditions.}

For Dirichlet-type $bc$, the adjoint boundary conditions $bc^*$
are also Dirichlet-type. Indeed, in view of (\ref{13}), $bc$ given
by a matrix (\ref{8}) are Dirichlet-type if and only if we have
 $z_1+ z_2=0, $ where $z_1 $ and $z_2 $ are the roots of
 (\ref{13}). By Lemma \ref{lem4} and the discussion after (\ref{31}), the
roots of the equation (\ref{13}) that corresponds to $bc^*$ are
$1/\overline{z_1},1/\overline{z_2},$  so we have  $$
\frac{1}{\overline{z_1}}+ \frac{1}{\overline{z_2}} =
\frac{\overline{z_1 +z_2}}{\overline{z_1 z_2}} =0. $$ Therefore,
$bc^*$ are Dirichlet-type also.\vspace{3mm}

4. {\em Regular but not strictly regular boundary conditions.}

Now we assume that (\ref{10}) holds, but (\ref{11}) fails, i.e.,
\begin{equation}
\label{68} (b+c)^2 - 4(bc-ad) = (b-c)^2 + 4ad =0.
\end{equation}
In this case the characteristic equation (\ref{13}) has one double
root:
\begin{equation}
\label{70} z_*=-(b+c)/2.
\end{equation}
Notice, that $z_* \neq 0 $ because otherwise (\ref{68}) would imply
$bc-ad =0$ which contradicts to the regularity condition (\ref{10}).

Let $\tau_* $ be chosen so that
\begin{equation}
\label{70.1}  z_* = -(b+c)/2= e^{i \pi \tau_*}, \quad   |
Re \, \tau_*|
\leq 1.
\end{equation}
Then all eigenvalues of $L^0_{bc}$ are given by
\begin{equation}
\label{70.2} Sp \, (L^0_{bc}) = \{ \tau_* +k,\;\; k \in 2\mathbb{Z}
\}.
\end{equation}
In view of Lemma \ref{lem2}, the corresponding eigenvectors have the
form (\ref{3}) with
$\begin{pmatrix} \xi \\ \eta \end{pmatrix} =\begin{pmatrix} \alpha_1 z_*\\
\alpha_2 \end{pmatrix},$ where $\begin{pmatrix} \alpha_1 \\ \alpha_2
\end{pmatrix}$ is an eigenvector of the matrix $A_{23}=
\left [ \begin{array}{cc} b  &  a\\ d  & c \end{array} \right ]$
corresponding to its double eigenvalue $-z_*,$ i.e.,
\begin{equation}
\label{71}
 (A_{23}+ z_* I)
\begin{pmatrix} \alpha_1 \\ \alpha_2 \end{pmatrix}=
\left [ \begin{array}{cc} b +z_* &  a\\ d  & c+z_*   \end{array}
\right ] \begin{pmatrix} \alpha_1 \\ \alpha_2
\end{pmatrix}= \left [ \begin{array}{cc} \frac{b-c}{2} & a\\ d &
\frac{c-b}{2}\end{array} \right ] \begin{pmatrix} \alpha_1 \\
\alpha_2 \end{pmatrix}=0.
\end{equation}

The matrix $A_{23}$ will have two linearly independent eigenvectors
$\begin{pmatrix} \alpha_1 \\ \alpha_2
\end{pmatrix}$ and $\begin{pmatrix} \beta_1 \\ \beta_2
\end{pmatrix}$ if and only if $A_{23}+ z_* I$
is the zero matrix, i.e.,
\begin{equation}
\label{72} b=c, \quad a=0, \quad d=0.
\end{equation}
Then the matrix (\ref{8}) has the form
\begin{equation}
\label{73} \left [ \begin{array}{cccc} 1 & b & 0 & 0\\ 0 & 0 & b &
1 \end{array} \right ], \qquad b\neq 0.
\end{equation}
We call the boundary conditions given by the matrix (\ref{8}) {\em
periodic--type} if (\ref{72}) holds, i.e., $bc $ is defined by
(\ref{73}).  Using the same argument as in the strictly regular case we get the
following lemma.

\begin{Lemma}
\label{lem5} For periodic-type $bc,$  (\ref{17})
with $\tau_1 = \tau_*$  gives all
eigenvalues of $L^0_{bc},$  and each eigenvalue is of geometric
multiplicity 2. There are linearly independent vectors
$\begin{pmatrix} \alpha_1
\\ \alpha_2
\end{pmatrix}$ and $\begin{pmatrix} \beta_1 \\ \beta_2
\end{pmatrix}$
such that the system $\Phi =\Phi^1 \cup \Phi^2,$  given by
(\ref{21}) and (\ref{22}) with $\tau_2=\tau_1= \tau_*, $ is a Riesz
basis in the space $L^2 (I, \mathbb{C}^2), \; I= [0,\pi].$ Its biorthogonal
system $\tilde{\Phi} =\tilde{\Phi}^{1} \cup \tilde{\Phi}^{2}$ is
defined by (\ref{21*}) and (\ref{22*}).
\end{Lemma}

Next we consider the case when (\ref{68}) holds but (\ref{72})
fails, i.e.,
\begin{equation}
\label{74} |b-c|+ |a| + |d| >0.
\end{equation}
As we will see below, in this case each eigenvalue of $L^0_{bc} $ is
of algebraic multiplicity 2 but of geometric multiplicity 1, i.e.,
associated eigenvectors appear. Here we have the following subcases:

(i) If $a= 0,$ then  (\ref{68}) implies $b=c,$ and by (\ref{74})
we have $d\neq 0.$ By the regularity condition (\ref{10}) we have
$ bc-ad \neq 0, $ which yields $b\neq 0. $ In other words, the
matrix (\ref{8}) has the form
\begin{equation}
\label{76} \left [ \begin{array}{cccc} 1 & b & 0 & 0\\ 0 & d & b &
1 \end{array} \right ], \qquad d, b\neq 0.
\end{equation}
Here we choose the following solution of (\ref{71})
\begin{equation}
\label{77} \alpha_1 =0, \quad \alpha_2 = d.
\end{equation}

(ii) If $d= 0,$ then  (\ref{68}) implies $b=c,$ and by (\ref{74})
we have $a\neq 0.$ Now the matrix (\ref{8}) has the form
\begin{equation}
\label{78} \left [ \begin{array}{cccc} 1 & b & a & 0\\ 0 & 0 & b &
1 \end{array} \right ], \qquad d, b\neq 0.
\end{equation}
Here we choose the following solution of (\ref{71}):
\begin{equation}
\label{79} \alpha_1 =a, \quad \alpha_2 =0.
\end{equation}

(iii) If $a,d\neq 0,$ then  (\ref{68}) implies $b\neq c.$ Here we
choose the following solution of (\ref{71}):
\begin{equation}
\label{80} \alpha_1 =a, \quad \alpha_2 = (c-b)/2.
\end{equation}

Of course, (\ref{70.2}) gives all eigenvalues. A corresponding
system of eigenvectors is
\begin{equation}
\label{83} \varphi^1_k=
\begin{pmatrix}  \alpha_1 e^{i \tau_* (\pi-x)} e^{-ikx}
\\ \alpha_2 e^{i \tau_* x} e^{ikx}\end{pmatrix},\quad k\in
2\mathbb{Z},
\end{equation}
where $\alpha_1, \alpha_2 $ are given, respectively, by (\ref{77}),
(\ref{79}) and (\ref{80}).

We look for a system of associated eigenvectors of the form
\begin{equation}
\label{85} \varphi^2_k=
\begin{pmatrix}    (\beta_1 - \alpha_1 x) e^{i \tau_* (\pi-x)} e^{-ikx}
\\  (\beta_2 + \alpha_2 x) e^{i \tau_* x} e^{ikx}\end{pmatrix},\quad k\in
2\mathbb{Z}.
\end{equation}
Then  $ L^0 \varphi^2_k = \lambda_k \varphi^2_k  - i \varphi^1_k, \;
$  so  $\varphi^2_k $ is an associated eigenvector if and only if it
satisfies the boundary conditions. This leads to the following system
of two linear equations in two unknowns $\beta_1$ and $\beta_2:$ $$
\begin{array}{ll} \beta_1 z_* + b(\beta_1 -\pi \alpha_1) + a \beta_2= 0
\\ d(\beta_1 -\pi \alpha_1)+ c \beta_2 + (\beta_2 + \pi \alpha_2)z_* =0,
\end{array}
 $$
or, equivalently, $$
\begin{array}{ll} (z_* +b) \beta_1  + a \beta_2= \pi b \alpha_1
\\ d\beta_1 + (c+z_*) \beta_2  =\pi d \alpha_1-\pi \alpha_2 z_*
\end{array}
 $$
By (\ref{70}), $b+ z_* = (b-c)/2, \; c + z_* = (c-b)/2. $ Moreover,
by (\ref{70}) and (\ref{71}), $d \alpha_1 = \alpha_2 (b-c)/2,$ and
therefore, $\pi d \alpha_1-\pi \alpha_2 z_*= \pi \alpha_2
\frac{b-c}{2} +\pi \alpha_2 \frac{b+c}{2}= \pi b \alpha_2.$

Thus, (\ref{85}) is a system of associated vectors if and only if
$\beta_1$ and $\beta_2$ satisfy
\begin{equation}
\label{86} \left [ \begin{array}{cccc}
b +z_* &  a\\ d  & c+z_* \end{array} \right ] \begin{pmatrix} \beta_1 \\
\beta_2 \end{pmatrix}= \left [ \begin{array}{cc} \frac{b-c}{2} & a\\
d & \frac{c-b}{2}\end{array} \right ]\begin{pmatrix} \beta_1 \\
\beta_2
\end{pmatrix}=\begin{pmatrix}  \pi b\alpha_1 \\ \pi b \alpha_2
\end{pmatrix}.
\end{equation}
Notice, that (\ref{71}) and (\ref{86}) mean that $\begin{pmatrix}
\alpha_1 \\ \alpha_2 \end{pmatrix}$ is an eigenvalue of the matrix
$A_{23} = \left [ \begin{array}{cc}  b &a \\ d & c \end{array} \right
] $ corresponding to
its double eigenvalue $-z_*, $ and $\begin{pmatrix} \beta_1 \\
\beta_2
\end{pmatrix}$ is an associated vector.

With $\alpha_1$ and $ \alpha_2 $ fixed, respectively, in (\ref{77}),
(\ref{79}) and (\ref{80}), we choose corresponding solutions of
(\ref{86}):
\begin{equation}
\label{87} \beta_1 =\pi b, \quad \beta_2 = 0 \quad \text{in the case
(i)};
\end{equation}
and
\begin{equation}
\label{90} \beta_1 =0, \quad \beta_2 = \pi b  \quad \text{in the
cases (ii) and (iii)}.
\end{equation}

\begin{Lemma}
\label{lem6} The system $\Phi $ of eigenfunctions $\varphi^1_k, \; k
\in 2\mathbb{Z}, $ and associated functions $\varphi^2_k, \; k \in
2\mathbb{Z}, $ given in (\ref{83}) and (\ref{85}), is a Riesz basis
in the space $L^2 (I,\mathbb{C}^2), \; I = [0,\pi].$ Its biorthogonal
system is $\tilde{\Phi} =\{ \tilde{\varphi}^1_k, \tilde{\varphi}^2_k,
\: k \in 2\mathbb{Z} \},$ where
\begin{equation}
\label{91} \tilde{\varphi}^1_k= \begin{pmatrix}
\bar{\Delta}^{-1} \overline{\alpha_2}
e^{i\overline{\tau_*} (\pi - x)} e^{-ikx} \\
\bar{\Delta}^{-1} \overline{\alpha_1}
e^{i\overline{\tau_*} x} e^{ikx}
\end{pmatrix}, \quad
\tilde{\varphi}^2_k= \begin{pmatrix} \bar{\Delta}^{-1}
[\overline{\beta_2} +\overline{\alpha_2} (\pi -x)]
e^{i\overline{\tau_*} (\pi - x)} e^{-ikx}
\\ \bar{\Delta}^{-1} [\overline{\beta_1} - \overline{\alpha_1}(\pi-x)]
e^{i\overline{\tau_*} x} e^{ikx}
\end{pmatrix}
\end{equation}
with $\Delta =\alpha_1 \beta_2 - \alpha_2 \beta_1 + \pi \alpha_1
\alpha_2. $
\end{Lemma}

\begin{proof}
Consider the operator $A: L^2 (I, \mathbb{C}^2) \to L^2 (I,
\mathbb{C}^2)$ defined by
\begin{equation}
\label{94} A
\begin{pmatrix}  f\\ g \end{pmatrix} =
\begin{pmatrix}  \alpha_1 e^{i \tau_* (\pi-x)} f(\pi -x)
\\  \alpha_2 e^{i \tau_* x} f(x) \end{pmatrix}
+ \begin{pmatrix}  (\beta_1 - \alpha_1 x) e^{i \tau_* (\pi-x)} g(\pi
-x)
\\  (\beta_2 + \alpha_2 x) e^{i \tau_* x} g(x) \end{pmatrix}.
\end{equation}
Since we have $\Phi = A(E),$ where $E$ is the orthonormal basis
(\ref{23}), the lemma will be proved if we show that $A$ is an
isomorphism. One can easily see that $A$ is bounded operator. Let
us find its inverse. By (\ref{94}), the equation $$A
\begin{pmatrix}  f\\ g \end{pmatrix} = \begin{pmatrix}  F\\ G
\end{pmatrix}$$ is equivalent to the following system of two
linear equations in two unknowns $f,g:$
\begin{eqnarray}
\label{95} \alpha_1  f(x) +(\beta_1 - \alpha_1 [\pi -x])   g(x) =
F(\pi -x)e^{-i \tau_* x},\\ \nonumber  \alpha_2  f(x) +(\beta_2 +
\alpha_2 x) g(x) = G(x)e^{-i \tau_* x}.
\end{eqnarray}
The determinant of this system is
$$ \Delta
= \det \left [
\begin{array}{cc} \alpha_1 & \beta_1-\pi \alpha_1+\alpha_1 x
\\ \alpha_2 & \beta_2 + \alpha_2x
\end{array} \right ]  =\begin{cases} -\pi bd &  \text{in case (i)}\\
\pi ab &  \text{in case (ii)}\\  \pi a (b+c)/2  &  \text{in case
(iii)} \end{cases} $$ due to our choices of $ \alpha_1, \alpha_2$ in
(\ref{77}),(\ref{79}),(\ref{80}) and  $\beta_1, \beta_2$ in
(\ref{87}) and (\ref{90}). Thus we get $$\Delta
\begin{pmatrix}  f(x)
\\   g(x) \end{pmatrix} =
\begin{pmatrix}  \left [(\beta_2+ \alpha_2 x)
F (\pi-x) -(\beta_1 - \pi \alpha_1 + \alpha_1 x)  G(x) \right ]
e^{-i \tau_* x}\\
\left [ -\alpha_2 F (\pi-x) + \alpha_1 G(x)\right ]e^{-i \tau_* x}
\end{pmatrix},$$ which implies (since $\Delta \neq 0 $)
\begin{equation}
\label{96} A^{-1}
\begin{pmatrix}  F\\ G \end{pmatrix} =\frac{1}{\Delta}
\begin{pmatrix} \left [(\beta_2+ \alpha_2 x) F (\pi-x) -
(\beta_1 - \pi \alpha_1 + \alpha_1 x)  G(x) \right ] e^{-i \tau_*
x}\\ \left [ -\alpha_2 F (\pi-x) + \alpha_1 G(x)\right ]e^{-i \tau_*
x}
\end{pmatrix}.
\end{equation}
Now it is easy to see
that the operator $A^{-1} $ is bounded.

A simple calculation
(similar to the one used in Lemma ~\ref{lem4})
shows
that the adjoint operator of $A^{-1}
$ is
\begin{equation}
\label{97} (A^{-1})^*\begin{pmatrix}  f\\ g \end{pmatrix}=
\bar{\Delta}^{-1}
\begin{pmatrix} [(\overline{\beta_2} +\overline{\alpha_2} \pi - \overline{\alpha_2}x)
f(\pi-x) +
\overline{\alpha_2} g(\pi -x)]e^{i\overline{\tau_*} (\pi - x)}  \\
[-(\overline{\beta_1} - \overline{\alpha_1}\pi+
\overline{\alpha_1}x) f(x) +\overline{\alpha_1}g(x)]
e^{i\overline{\tau_*} x}
\end{pmatrix}.
\end{equation}
Since we have $\tilde{\Phi} =(A^{-1})^*(E), $ where $E$ is the
orthonormal  basis defined in (\ref{23}), the family $ \tilde{\Phi} $
is the biorthogonal system to $\Phi$.
\end{proof}

\section{Matrix representation of $L_{bc}$ and its resolvent $R_{bc}(\lambda)$}

Next we consider, for arbitrary regular $bc,$
the Fourier representation of $L_{bc}$ and its
resolvent $L_{bc}(\lambda)$ with respect to a corresponding Riesz basis consisting
of eigenfunctions and associated functions of the operator
$L_{bc}^0 $  (constructed in Lemmas \ref{lem3}, \ref{lem5},  {\ref{lem6}).\vspace{3mm}

1.  Let $V:  L^2 (I, \mathbb{C}^2) \to L^2 (I, \mathbb{C}^2)$ be the operator of
multiplication by the matrix $v(x) = \left [
\begin{array}{cc}0&P(x)\\Q(x)&0
\end{array} \right ],$  i.e.,
$$ V \begin{pmatrix}y_1 \\ y_2  \end{pmatrix} =\left [
\begin{array}{cc}0&P(x)\\Q(x) & 0
\end{array} \right ] \begin{pmatrix}y_1 \\ y_2 \end{pmatrix}=
 \begin{pmatrix} Py_2 \\ Qy_1 \end{pmatrix}.
$$
For a regular boundary condition $bc,$  let $\Phi=
\{\varphi^1_k,\varphi^2_k, k\in \mathbb{Z}\} $ and $
\tilde{\Phi}=\{\tilde{\varphi}^1_k,\tilde{\varphi}^2_k, k\in
\mathbb{Z}\}$ be the corresponding Riesz basis (consisting of
eigenfunctions and associated functions of the operator $L_{bc}^0$)
and its biorthogonal system constructed,  respectively, in
Lemma~\ref{lem3} if $bc$ is strictly regular, in Lemma~\ref{lem5} if
$bc$ is periodic type, and in Lemma~\ref{lem6} otherwise. In this
section and thereafter, we consider matrix representation with
respect to that basis only.

\begin{Lemma}
\label{lem10} The matrix representation of $V$ with respect to the
basis $\Phi$ has the form
\begin{equation}\label{m1}
V \sim \left [ \begin{array}{cc} V^{11} &V^{12}\\V^{21}&V^{22}
\end{array} \right ], \quad  V^{\mu\nu}= \left ( V^{\mu\nu}_{jk} \right
)_{j,k\in 2\mathbb{Z}}, \quad \mu, \nu \in \{1,2\},
\end{equation}
\begin{equation}\label{m2}
 V^{\mu\nu}_{jk} =\langle V\varphi^\nu_k, \tilde{\varphi}^\mu_j \rangle =
 w^{\mu\nu} (j+k),
\end{equation}
where
\begin{equation}\label{m3}
w^{\mu\nu}=\left
 (w^{\mu\nu}(m)\right ) \in \ell^2 (2\mathbb{Z}), \quad
 \|w^{\mu\nu}\|_{\ell^2} \leq C (\|P\|_{L^2}+\|Q\|_{L^2}),
 \end{equation}
with $C=C(\Phi,\tilde{\Phi}).$
\end{Lemma}

\begin{proof}
We consider only the case where $\mu =1, \nu =2 $  because the proof is
similar in the other three cases.

If $bc$ is strictly regular, then we get, by
(\ref{21})--(\ref{22*}),
$$
V^{12}_{jk} =\langle V\varphi^2_k, \tilde{\varphi}^1_j \rangle= \left
\langle \begin{pmatrix} P(x)\beta_2 e^{i\tau_2 x}e^{ikx}\\
Q(x)\beta_1 e^{i\tau_2 (\pi - x)}e^{-ikx}
\end{pmatrix},  \begin{pmatrix}  \overline{\alpha^\prime_1}
e^{i \overline{\tau_1} (\pi-x)} e^{-ijx}\\
\overline{\alpha^\prime_2} e^{i \overline{\tau_1} x}
e^{ijx}\end{pmatrix} \right \rangle
$$
$$
=\frac{1}{\pi} \int_0^\pi \left [\alpha^\prime_1 \beta_2 e^{i(\tau_2
x+\tau_1 (x- \pi ))} P(x) e^{i(j+k)x} dx + \alpha^\prime_2 \beta_1
e^{i(\tau_2 (\pi - x)-\tau_1 x)} Q(x) e^{-i(j+k)x} \right ] dx.
$$
Therefore, (\ref{m2}) holds for $\mu =1, \nu =2 $ with
\begin{equation}\label{m4}
w^{12}(m): = p^{12}(-m)+ q^{12}(m), \;\; m \in 2\mathbb{Z},
 \end{equation}
where $p^{12}(m)$ and  $ q^{12}(m), \; m \in
2\mathbb{Z},$  are the Fourier coefficients of the functions
$g^{12}(x)P(x)$ and $h^{12}(x)Q(x),$ with
$$
g^{12}(x):=\alpha^\prime_1 \beta_2 e^{i(\tau_2 x+\tau_1 (x- \pi ))},
\quad h^{12}(x):= \alpha^\prime_2 \beta_1 e^{i(\tau_2 (\pi -
x)-\tau_1 x)} .
$$
By the Parseval identity,
$$
\sum_m |p^{12}(m)|^2 = \|g^{12}(x)P(x)\|^2_{L^2 (I)} \leq
\sup_{[0,\pi]} |g^{12}(x)|^2 \cdot \|P\|^2_{L^2 (I)}
$$
and
$$ \sum_m |q^{12}(m)|^2 = \|h^{12}(x)Q(x)\|^2_{L^2 (I)} \leq
\sup_{[0,\pi]} |h^{12}(x)|^2  \cdot \|Q\|^2_{L^2 (I)}.
$$
Thus, (\ref{m3}) holds with a constant $C$ depending on the
parameters $\alpha^\prime_1, \alpha^\prime_2, \beta_1, \beta_2,
\tau_1, \tau_2. $

The proof is exactly the same if $bc$ is periodic type (the same
formulas work but with $\tau_2=\tau_1=\tau_*$).

If $bc$ is not strictly regular and not of periodic type, then by
(\ref{85}) and (\ref{91}) we have
$$
\langle V\varphi^2_k,\tilde{\varphi}^1_k \rangle =\left \langle
\begin{pmatrix} P(x)(\beta_2 + \alpha_2 x) e^{i \tau_* x} e^{ikx}\\
 Q(x) (\beta_1 - \alpha_1 x) e^{i \tau_* (\pi-x)} e^{-ikx}
\end{pmatrix},  \begin{pmatrix}
\bar{\Delta}^{-1} \overline{\alpha_2}
e^{i\overline{\tau_*} (\pi - x)} e^{-ijx} \\
\bar{\Delta}^{-1} \overline{\alpha_1}
e^{i\overline{\tau_*} x} e^{ijx}
\end{pmatrix} \right \rangle
$$
$$
=\frac{1}{\pi} \int_0^\pi \left [
P(x)\frac{\alpha_2}{\Delta}(\beta_2 + \alpha_2 x) e^{i \tau_*
(2x-\pi)} e^{i(j+k)x}\right ] dx
$$
$$
+\frac{1}{\pi} \int_0^\pi  \left [Q(x)\frac{\alpha_1}{\Delta}
(\beta_1 - \alpha_1 x) e^{i \tau_* (\pi-2x)} e^{-i(j+k)x}
 \right ]dx.
$$

Therefore, (\ref{m2}) holds for $\mu =1, \nu =2 $ with
\begin{equation}\label{m7}
w^{12}(m): = p_1^{12}(-m)+ q_1^{12}(m), \;\; m \in 2\mathbb{Z},
 \end{equation}
where $p_1^{12}(m)$ and  $ q_1^{12}(m), \; m
\in 2\mathbb{Z},$ are the Fourier coefficients of the functions
$g_1^{12}(x)P(x)$ and $h_1^{12}(x)Q(x),$ with
$$
g_1^{12}(x):=\frac{\alpha_2}{\Delta}(\beta_2 + \alpha_2 x) e^{i
\tau_* (2x-\pi)}, \quad h_1^{12}(x):= \frac{\alpha_1}{\Delta}
(\beta_1 - \alpha_1 x) e^{i \tau_* (\pi-2x)}.
$$
Since these functions are bounded, again the Parseval identity
implies (\ref{m3}) with a constant $C$ depending on parameters
$\alpha_1, \alpha_2, \beta_1, \beta_2, \tau_*. $

\end{proof}

2. If $bc$ is strictly regular boundary condition, then  by (\ref{17})
and (\ref{18})  the spectrum of $L_{bc}^0$ consists of two disjoint
sequences
$$
Sp (L_{bc}^0) = \{\tau_1 +k, \; k \in 2\mathbb{Z} \} \cup \{\tau_2
+k, \; k \in 2\mathbb{Z} \}. $$
The resolvent operator $R_{bc}^0 (\lambda) = (\lambda
-L_{bc}^0)^{-1}$ is well defined for $\lambda \not\in Sp (L_{bc}^0),
$ and we have
\begin{equation}
\label{1.10} R_{bc}^0 (\lambda) \varphi^\mu_k = \frac{1}{\lambda -
\tau_\mu -k}\varphi^\mu_k, \quad   k \in 2\mathbb{Z}, \;\;
\mu=1,2.
\end{equation}

By (\ref{70.2}), for regular but not strictly regular $bc$ the
spectrum of $L_{bc}^0$ is given by
$$
Sp \, (L^0_{bc}) = \{ \tau_* +k,\;\; k \in 2\mathbb{Z} \},
$$
where each eigenvalue is of algebraic multiplicity 2.
The resolvent operator $R_{bc}^0 (\lambda) = (\lambda
-L_{bc}^0)^{-1}$ is well defined for $\lambda \not\in Sp (L_{bc}^0)
$ by
\begin{equation}
\label{1.10a} R_{bc}^0 (\lambda) \varphi^\mu_k = \frac{1}{\lambda -
\tau_* -k}\varphi^\mu_k, \quad   k \in 2\mathbb{Z}, \;\; \mu=1,2.
\end{equation}

The standard perturbation formula for the resolvent
$$
R_{bc} (\lambda) =  R_{bc}^0 (\lambda) +R_{bc}^0 (\lambda)VR_{bc}^0
(\lambda)+R_{bc}^0 (\lambda)VR_{bc}^0 (\lambda) VR_{bc}^0 (\lambda)+
\cdots $$  can be written as
\begin{equation}
\label{1.11} R_{bc} (\lambda)= (K_\lambda)^2 + \sum_{s=1}^\infty
K_\lambda (K_\lambda V K_\lambda)^s K_\lambda
\end{equation}
provided
\begin{equation}
\label{1.12} (K_\lambda)^2 = R_{bc}^0 (\lambda).
\end{equation}
Then the operator $R_{bc}(\lambda) $ is well-defined by
(\ref{1.11}) if
\begin{equation}
\label{1.13} \|K_\lambda V K_\lambda \|  <1.
\end{equation}
In the next
section we will give conditions under which (\ref{1.13}) holds.

In view of (\ref{1.10}) and (\ref{1.10a}),  we define an operator
$K= K_\lambda $ with the property (\ref{1.12}), respectively,  for
strictly regular $bc$ by
\begin{equation}
\label{1.14} K_\lambda \varphi^\mu_k  = \frac{1}{\sqrt{\lambda
-\tau_\mu -k} } \varphi^\mu_k,  \quad  k \in 2\mathbb{Z}, \;\;
\mu=1,2,
\end{equation}
and for regular but not strictly regular $bc$ by
\begin{equation}
\label{1.14a} K_\lambda \varphi^\mu_k  = \frac{1}{\sqrt{\lambda
-\tau_* -k} } \varphi^\mu_k,  \quad  k \in 2\mathbb{Z}, \;\;
\mu=1,2,
\end{equation}
where $$\sqrt{z}= \sqrt{r} e^{i\varphi/2} \quad \mbox{if} \quad z=
re^{i\varphi}, \;\; -\pi \leq \varphi < \pi. $$

By (\ref{m1}), (\ref{m2}), (\ref{1.14}) and (\ref{1.14a}), we have
\begin{equation}
\label{1.15} \langle K_\lambda V K_\lambda \varphi^\nu_k,
\tilde{\varphi}^\mu_j \rangle =
 \frac{w^{\mu\nu} (j+k)}{\sqrt{\lambda - \tau_\mu
 -j}\,\sqrt{\lambda - \tau_\nu -k}}, \quad j,k \in 2\mathbb{Z}
\end{equation}
for strictly regular  $bc,$  and
\begin{equation}
\label{1.15a} \langle K_\lambda V K_\lambda \varphi^\nu_k,
\tilde{\varphi}^\mu_j \rangle =
 \frac{w^{\mu\nu} (j+k)}{\sqrt{\lambda - \tau_*
 -j}\,\sqrt{\lambda - \tau_* -k}}, \quad j,k \in 2\mathbb{Z}
\end{equation}
for regular but not strictly regular $bc.$

Therefore, for $s \geq 1, $ it follows that
\begin{equation}
\label{1.16} \langle K_\lambda  (K_\lambda V K_\lambda)^s K_\lambda
\varphi^\nu_k, \tilde{\varphi}^\mu_m \rangle
\end{equation}
$$ =\sum_{\gamma_1,..,\gamma_s=1}^2 \sum_{j_1, \ldots j_s}
\frac{w^{\mu \gamma_1}(m + i_1) w^{\gamma_1 \gamma_2}(i_1+i_2)
\cdots  w^{\gamma_{s-1} \gamma_s}(i_{s-1}+i_s) w^{\gamma_s
\nu}(i_s +k)} {(\lambda -\tau_{\mu}-m)(\lambda
-\tau_{\gamma_1}-i_1) \cdots (\lambda -\tau_{\gamma_s}-i_s
)(\lambda -\tau_{\nu}-k) } $$ for strictly regular  $bc,$  and
\begin{equation}
\label{1.16a} \langle K_\lambda  (K_\lambda V K_\lambda)^s K_\lambda
\varphi^\nu_k, \tilde{\varphi}^\mu_j \rangle
\end{equation}
$$
=\sum_{\gamma_1,..,\gamma_s=1}^2 \sum_{j_1, \ldots j_s} \frac{w^{\mu
\gamma_1}(m + i_1)  w^{\gamma_1 \gamma_2}(i_1+i_2)\cdots  w^{\gamma_{s-1} \gamma_s}(i_{s-1}+i_s) w^{\gamma_s
\nu}(i_s +k)} {(\lambda -\tau_*-m)(\lambda -\tau_*-i_1) \cdots
(\lambda -\tau_*-i_s )(\lambda -\tau_*-k) }.
$$
for regular but not strictly regular $bc.$ In view of (\ref{1.11}),
the formulas (\ref{1.16}) and (\ref{1.16a}) determine the matrix
representation of the resolvent $R_{bc} (\lambda).$

\section{Localization of spectra}

In this section we consider the spectra localization of the
operators $L_{bc} = L_{bc}^0 + V,$ where $V $
denotes the operator of multiplication by the matrix
$v(x) = \begin{pmatrix} 0 & P(x) \\ Q(x) & 0 \end{pmatrix}. $
\vspace{3mm}

1. In view of (\ref{1.15}) and (\ref{1.15a}),  the Hilbert--Schmidt
norm of the operator $K_\lambda V K_\lambda$ with respect to the
Riesz basis $\Phi $  (see (\ref{p4}))  is given by
\begin{equation}
\label{2.17} (\|K_\lambda V K_\lambda \|^*_{HS})^2 =
\sum_{\nu,\mu=1}^2 \sum_{j,k\in 2\mathbb{Z}} \frac{|w^{\mu\nu}
(j+k)|^2}{|\lambda - \tau_\mu
 -j||\lambda - \tau_\nu -k|}
\end{equation}
for regular $bc,$  and
\begin{equation}
\label{2.17a} (\|K_\lambda V K_\lambda \|^*_{HS})^2 =
\sum_{\nu,\mu=1}^2 \sum_{j,k\in 2\mathbb{Z}} \frac{|w^{\mu\nu}
(j+k)|^2}{|\lambda - \tau_*
 -j||\lambda - \tau_* -k|}
\end{equation}
for regular but not strictly regular $bc.$

For convenience, we set
\begin{equation}
\label{2.18} r(m) = \max \{|w^{\mu\nu} (m)|, \; \mu, \nu =1,2 \},
\quad m \in 2\mathbb{Z};
\end{equation}
then
\begin{equation}
\label{2.18a} r= (r(k)) \in \ell^2 (2\mathbb{Z}), \quad \|r\|\leq C
(\|P\|_{L^2}+ \|Q\|_{L^2}),
\end{equation}
where $C=C(bc).$

Now we define operators $\bar{V} $ and $\bar{K}_\lambda $ which
matrix representations
dominate, respectively, the matrix representations of
 $V$ and $K_\lambda,$ as follows:
\begin{equation}
\label{2.19} \bar{V}\varphi_n^\mu  = \sum_{k  \in  2\mathbb{Z}}
r(k+n) (\varphi_k^1+\varphi_k^2), \quad \mu = 1,2; \;\; n\in
2\mathbb{Z},
\end{equation}
\begin{equation}
\label{2.20} \bar{K}_\lambda \varphi_n^\mu = \frac{1}{\sqrt{|\lambda
-\tau_\mu - n|}} \varphi_n^\mu, \quad \mu = 1,2; \;\; n\in
2\mathbb{Z}
\end{equation}
for strictly regular $bc,$ and
\begin{equation}
\label{2.20a} \bar{K}_\lambda \varphi_n^\mu =
\frac{1}{\sqrt{|\lambda -\tau_* - n|}} \varphi_n^\mu, \quad \mu =
1,2; \;\; n\in 2\mathbb{Z}
\end{equation}
for regular but not strictly regular $bc.$

The matrix elements of the operator $K_\lambda V K_\lambda$ do not
exceed, by absolute value, the matrix elements of $\bar{K}_\lambda
\bar{V} \bar{K}_\lambda.$ Therefore, in view of (\ref{2.17}) -- (\ref{2.18})
and Lemma~\ref{lem10}, it follows that
\begin{equation}
\label{2.21}
 (\|K_\lambda V K_\lambda\|^*_{HS})^2  \leq
(\|\bar{K}_\lambda \bar{V} \bar{K}_\lambda\|_{HS}^*)^2 =
 \sum_{\mu,\nu=1}^2 \sum_{j,k\in 2\mathbb{Z}}
\frac{|r(j+k)|^2}{|\lambda -\tau_\mu - j||\lambda -\tau_\nu - k|},
\end{equation}
for regular  $bc,$ and
\begin{equation}
\label{2.21a}
 (\|K_\lambda V K_\lambda\|^*_{HS})^2  \leq
(\|\bar{K}_\lambda \bar{V} \bar{K}_\lambda\|_{HS}^*)^2 =
 4 \sum_{j,k\in 2\mathbb{Z}}
\frac{|r(j+k)|^2}{|\lambda -\tau_* - j||\lambda -\tau_* - k|},
\end{equation}
for regular but not strictly regular $bc.$

For each $\ell^2$--sequence  $x=(x(j))_{j \in \mathbb{Z}} $ and $m
\in \mathbb{N}$ we set
\begin{equation}
\label{2.22} \mathcal{E}_m (x) = \left (   \sum_{|j|\geq m} |x(j)|^2
\right )^{1/2}.
\end{equation}
Next we consider separately the case of strictly regular $bc$ and
the case of regular but not strictly regular $bc.$ \vspace{2mm}

2. {\em Strictly regular} $bc.$ We subdivide the complex plane
$\mathbb{C}$ into strips
\begin{equation}
\label{2.24} H_m = \left \{z\in \mathbb{C}: \; -1 \leq Re \left
(z-m- \frac{\tau_1 +\tau_2}{2}  \right ) \leq 1  \right \}, \quad m
\in 2\mathbb{Z},
\end{equation}
and set
\begin{equation}
\label{2.25} H^N = \bigcup_{|m|\leq N}    H_m
\end{equation}
\begin{equation}
\label{2.26}  R_{NT} = \left \{z= x+it: \;  \;
 \left |  x-Re \, \frac{\tau_1 +\tau_2}{2}   \right | < N+1,\;
   |t| < T \right \},
\end{equation}
where  $N \in 2\mathbb{N}$  and
\begin{equation}
\label{2.28a}     T = 2 \max \left ( |Im \, \tau_1|, |Im \,
\tau_2|, 384 \|A\| \|A^{-1}\| \|r\|^2 \right )
\end{equation}
with $A$ being the isomorphism defined by
(\ref{24}).

Let
\begin{equation}
\label{2.29}  \rho: = \min ( 1-|Re (\tau_1 -\tau_2)|/2,\, |\tau_1
-\tau_2|/2),
\end{equation}
and
\begin{equation}
\label{2.30} D_m^\mu =\{z\in \mathbb{C}: \;\; |z-\tau_\mu -m| < \rho
\}, \quad m \in 2\mathbb{Z}.
\end{equation}

\begin{Lemma}
\label{loc1} (a) In the above notations, the following estimates hold:

(a)  if $\lambda \in H_m
\setminus ( D^1_m \cup D^2_m
 ), \; m\neq 0, $ then
\begin{equation}
\label{2.31}
 \sum_{\mu,\nu=1}^2 \sum_{j,k\in 2\mathbb{Z}}
\frac{|r(j+k)|^2}{|\lambda -\tau_\mu - j||\lambda -\tau_\nu - k|}
\leq  \left  ( \frac{30}{\rho} \right )^2
\left ( \frac{\|r\|^2}{\sqrt{|m|}} + (\mathcal{E}_{|m|}
(r))^2 \right );
\end{equation}

(b) if $\lambda \in H^N \setminus R_{NT},$ then
\begin{equation}
\label{2.32}
 \sum_{\mu,\nu=1}^2 \sum_{j,k\in 2\mathbb{Z}}
\frac{|r(j+k)|^2}{|\lambda -\tau_\mu - j||\lambda -\tau_\nu - k|}
\leq  \frac{384}{T} \|r \|^2.
\end{equation}
\end{Lemma}

\begin{proof}
(a) If $\lambda \in H_m $  then
\begin{equation}
\label{2.33} |\lambda - \tau_\mu - j| \geq |m-j|/4,
\quad   j \in 2\mathbb{Z} \setminus \{m\}, \;\; \mu =1,2.
\end{equation}
Indeed,  $|m-j| \geq 2, $  so  (\ref{2.24}) and (\ref{15}) imply
$$
|Re \, (\lambda - \tau_\mu - j)|
 \geq |m-j| -1-  \frac{1}{2} |Re \, (\tau_1- \tau_2)|  \geq
|m-j|-\frac{3}{2} \geq \frac{1}{4} |m-j|.
$$

In view of (\ref{2.33}), the sum in (\ref{2.31}) does not exceed
$$ 4^3 \sum_{j,k\neq m} \frac{|r(j+k|^2}{|m-j| |m-k|} + 4^2
\sum_{k\neq m} \frac{|r(m+k|^2}{\rho|m-k|}+4^2 \sum_{j\neq m}
\frac{|r(j+m|^2}{|m-j|\rho} + 4\frac{|r(2m)|^2}{\rho^2}. $$

Now the estimate (\ref{2.31}) follows from the inequalities
(\ref{t1}) and (\ref{t2}) below.
\begin{Lemma}
\label{lemt1} If $r = (r(k)) \in \ell^2 (2\mathbb{Z}), $  then
\begin{equation}
\label{t1} \sum_{k\neq n} \frac{|r(n+k)|^2}{|n-k|} \leq
\frac{\|r\|^2}{|n|} + (\mathcal{E}_{|n|} (r))^2, \quad |n| \geq 1;
\end{equation}
\begin{equation}
\label{t2} \sum_{i,k\neq n} \frac{|r(i+k)|^2}{|n-i||n-k|} \leq 12
\left ( \frac{\|r\|^2}{\sqrt{|n|}} + (\mathcal{E}_{|n|} (r))^2
\right ),  \quad |n| \geq 1,
\end{equation}
\end{Lemma}
Lemma \ref{lemt1} is identical to Lemma 7 in \cite{DM20}; a proof
is provided there.

Next we prove (\ref{2.32}).  If $\lambda \in H^N \setminus R_{NT},
$
 then $\lambda \in H_m$   for some even integer $m \in [-N,N],$
and we have
\begin{equation}
\label{2.34} | \lambda - \tau_\mu - j| \geq \frac{1}{4\sqrt{2}}
(|j-m| +T), \quad  \mu=1,2; \;\; j \in 2\mathbb{Z}.
\end{equation}
Indeed, $\lambda \in H_m \setminus R_{NT} $  means that $$ \lambda
= m+  Re \, \frac{\tau_1 + \tau_2}{2} + \xi + i \eta \quad
\text{with} \;\; \xi, \eta  \in \mathbb{R}, \; \; |\xi|\leq 1, \;
|\eta| \geq T. $$ Therefore, if $j=m, $ then by (\ref{2.28a}) we
obtain $$ | \lambda - \tau_\mu - j |  \geq  |Im \,  ( \lambda -
\tau_\mu - j ) |   \geq T - |Im \, \tau_\mu | \geq T/2,
 $$ so (\ref{2.34}) holds. Otherwise,
$|j-m| \geq 2$  (so $|j-m|-3/2 \geq |j-m|/4$); then  by the
inequality $|x+iy| \geq  \frac{1}{\sqrt{2}} |x|
+\frac{1}{\sqrt{2}} |y| $ and (\ref{15}) we obtain $$ | \lambda -
\tau_\mu - j | \geq  \frac{1}{\sqrt{2}}\left (|j-m| -  \left | Re
\, \frac{\tau_1 - \tau_2}{2}  \right | -1 \right ) +
 \frac{1}{\sqrt{2}} ( T-  |Im \,   \tau_\mu | )
 $$
 $$
  \geq   \frac{1}{\sqrt{2}}
  ( |j-m|-3/2) + \frac{1}{2\sqrt{2}} T  \geq \frac{1}{4\sqrt{2}} (|j-m|+T).
$$

In view of  (\ref{2.34}), the sum in (\ref{2.32}) does not exceed
$$ \sigma :=  128 \sum_{j,k \in 2\mathbb{Z}}  \frac{| r(j+k)
|^2}{(|j-m|+T)(|k-m|+T)} . $$ By the Cauchy inequality, $$ \sigma
\leq    128\left (  \sum_{j,k \in 2\mathbb{Z}}  \frac{| r(j+k)
|^2}{(|j-m|+T)^2}  \right )^{1/2} \left (  \sum_{j,k \in
2\mathbb{Z}}  \frac{| r(j+k) |^2}{(|k-m|+T)^2}  \right )^{1/2}. $$
Since $$ \sum_{j \in 2\mathbb{Z}}  \frac{1}{(|j-m|+T)^2}  \leq
\frac{1}{T^2} + 2\int_0^\infty  \frac{1}{(x+T)^2} dx
=\frac{1}{T^2} + \frac{2}{T}  \leq \frac{3}{T}, $$ it follows that
$\sigma \leq  \frac{384}{T} \|r\|^2,$   which completes the proof.
\end{proof}

\begin{Theorem}
\label{srl} In the above notations, for each strictly regular $bc$
 there is an $N = N(v, bc) \in 2\mathbb{N}$ such that
\begin{equation}
\label{2.41} Sp\, (L_{bc} (v_{\zeta}) \subset R_{NT} \cup \bigcup_{|m|>N}
 \left (
D^1_m \cup D^2_m \right ) \quad
\text{for} \;\; v_{\zeta}= \zeta v, \; |\zeta| \leq 1.
\end{equation}

\end{Theorem}

\begin{proof}
Let $G$ be the set in the right-hand side of (\ref{2.41}). In order
to prove (\ref{2.41}) for  $\zeta = 1$, it is enough to explain that
the resolvent $R_\lambda (v)  = (\lambda- L(v))^{-1} $ is
well-defined for $\lambda \in \mathbb{C}\setminus G. $

In view of (\ref{1.11}) -- (\ref{1.13}), $R_\lambda (v)$ is
well-defined if $\|K_\lambda V K_\lambda \|< 1.$ From Lemma
\ref{lem1},  formula (\ref{2.21}), Lemma~\ref{loc1} and the choice
(\ref{2.28a}) of  the constant $T$ it follows that
\begin{equation}
\label{2.410}
\|K_\lambda
VK_\lambda \| \leq \|K_\lambda VK_\lambda \|_{HS} \leq \|A\|
\|A^{-1}\| \|K_\lambda VK_\lambda \|^*_{HS} <1 \;\; \text{for}
\;\;  \lambda \in \mathbb{C}\setminus G
\end{equation}
if $N $ is chosen so
large that  the right-hand sides of (\ref{2.31}) (for $|m|>N)$ and
(\ref{2.32}) are strictly less than 1. In view of Lemma
~\ref{lem10} and (\ref{2.18}), (\ref{2.410})
holds for $\zeta v, \; |\zeta | \leq 1 $ as well.
Therefore, (\ref{2.41}) holds with $N= N(v, bc).
$

\end{proof}

3. {\em Regular but not strictly regular boundary conditions.}  Now we
subdivide the complex plane $\mathbb{C}$ into strips
\begin{equation}
\label{2.24a} H_m = \left \{z\in \mathbb{C}: \; -1 \leq Re \left
(z-m- \tau_*  \right ) \leq 1  \right \}, \quad m \in 2\mathbb{Z},
\end{equation}
and set
\begin{equation}
\label{2.25a} H^N = \bigcup_{|m|\leq N}  H_m,
\end{equation}
\begin{equation}
 \label{2.26a}  R_{NT} = \left \{z= x+it: \;  \;  \left |  x-Re \, \tau_*   \right | < N+1,\;     |t| < T \right \},
\end{equation}
where $N \in 2\mathbb{N}$ and
\begin{equation}
\label{2.28c}      T = 2 \max \left ( | Im \, \tau_*|,  96 \|A\| \|A^{-1}\| \|r\|^2 \right )
\end{equation}
with $A$ being the isomorphism defined by
(\ref{24})  (for periodic type boundary conditions)
and (\ref{94}) otherwise.

Let
\begin{equation}
\label{2.30a} D_m =\{z\in \mathbb{C}: \;\; |z-\tau_* -m| < 1/4 \},
\quad m \in 2\mathbb{Z}.
\end{equation}

\begin{Lemma}
\label{loc2} (a) In the above notations, if $\lambda \in H_m
\setminus  D_m
 ), \; m\neq 0, $ then
\begin{equation}
\label{2.31b}
 \sum_{j,k\in 2\mathbb{Z}}
\frac{|r(j+k)|^2}{|\lambda -\tau_* - j||\lambda -\tau_* - k|} \leq C
\left ( \frac{\|r\|^2}{\sqrt{|m|}} + (\mathcal{E}_{|m|} (r))^2
\right ),
\end{equation}
where $C$  is an absolute constant;

(b) if $\lambda \in H^N \setminus R_{NT},$ then
\begin{equation}
\label{2.32b} \sum_{j,k\in 2\mathbb{Z}} \frac{|r(j+k)|^2}{|\lambda
-\tau_* - j||\lambda -\tau_* - k|} \leq  \frac{24}{T} \|r \|^2.
\end{equation}
\end{Lemma}

\begin{proof}
If $\lambda \in H_m $  then (compare with (\ref{2.33}))
\begin{equation}
\label{2.33b} |\lambda - \tau_* - j| \geq |m-j|/4 \quad j \neq m, \;
j \in \mathbb{Z}.
\end{equation}
Therefore, the sum in (\ref{2.31b}) does not exceed
$$
4^2 \sum_{j,k\neq m} \frac{|r(j+k|^2}{|m-j||m-k|} + 4^2 \sum_{k\neq
m} \frac{|r(j+k|^2}{|m-k|}+4^2 \sum_{j\neq m}
\frac{|r(j+k|^2}{|m-j|} + 4^2 |r(2m)|^2.
$$

Now the estimate (\ref{2.31b}) follows from the inequalities
(\ref{t1}) and (\ref{t2}) in Lemma~\ref{lemt1}.

Next we prove (\ref{2.32b}).  If $\lambda \in H^N \setminus R_{NT}, $
 then $\lambda \in H_m$   for some integer $m \in [-N,N];$
then  (compare with (\ref{2.34})) we have
\begin{equation}
\label{2.34b} | \lambda - \tau_* - j| \geq \frac{1}{2\sqrt{2}} (|j-m| +T),
\quad  \mu=1,2; \;\; j \in 2\mathbb{Z}.
\end{equation}
The proof of (\ref{2.34b})  is similar to the proof of
 (\ref{2.34}), and therefore,  it is omitted.
Moreover, using (\ref{2.34b}) one can complete
the proof of part (b) exactly as it is
done in the proof of Lemma~\ref{loc1}.

\end{proof}

\begin{Theorem}
\label{rl} In the above notations, for each regular but not strictly
regular $bc$
 there is $N = N(v, bc) \in 2\mathbb{N}$ such that
\begin{equation}
\label{2.41a} Sp\, (L_{bc} (v_\zeta) \subset R_{NT} \cup \bigcup_{|m|>N}
D_m \quad \text{for} \;\; v_\zeta = \zeta v, \;\; |\zeta| \leq 1.
\end{equation}
\end{Theorem}

\begin{proof}
We follow the proof of Theorem \ref{srl}  
but use instead of (\ref{2.21}),  Lemma~\ref{loc1} and (\ref{2.28a})
their counterparts (\ref{2.21a}),  Lemma~\ref{loc2} and (\ref{2.28c}).
We omit further details.
\end{proof}

\section{Bari--Markus property in the case
of strictly regular boundary conditions}

We use the notations of the previous section.  For strictly regular  $bc$
 Theorem \ref{srl}  gives  the following localization of
 the spectrum of the Dirac operator   $L_{bc}:$
 $$ Sp\, (L_{bc}) \subset R_{NT} \cup \bigcup_{|n|>N} \left (
D^1_n \cup D^2_n \right ). $$

Let us consider the
Riesz projections associated with $L_{bc}$
\begin{equation}
\label{3.1} S_N = \frac{1}{2\pi i} \int_{\partial R_{NT}} (\lambda -
L)^{-1} d\lambda, \quad P_{n,\alpha} = \frac{1}{2\pi i} \int_{\partial
D_n^\alpha} (\lambda - L)^{-1} d\lambda, \quad  \alpha=1,2,
\end{equation}
and let $S^0_N$ and $P_{n,\alpha}^0$ be the Riesz projections
associated with the free operator $L^0_{bc}.$

\begin{Theorem}
\label{thm1} Suppose $L_{bc}$ and $L_{bc}^0$ are, respectively,
the Dirac operator with an $L^2$ potential $v$ and the
corresponding free Dirac operator, subject to the same strictly regular
boundary conditions $bc.$ Then, there is an $N \in 2\mathbb{N}$
such that the Riesz projections $S_N, \,P_{n,\alpha}$ and $ S_N^0,
\, P_{n,\alpha}^0, $ $ n\in 2\mathbb{Z}, \;|n|>N, \alpha=1,2, $ associated with
$L$ and $L^0$ are well defined by (\ref{3.1}), and we have
\begin{equation}
\label{3.2}  \dim P_{n,\alpha}= \dim P_{n,\alpha}^0= 1, \quad \dim
S_N = \dim S_N^0 =2N;
\end{equation}
\begin{equation}
\label{3.3}   \sum_{|n|>N} \|P_{n,\alpha} - P_{n,\alpha}^0\|^2 <
\infty, \quad  \alpha =1,2.
\end{equation}
Moreover,  the system $\{S_N, \;P_{n,\alpha}, \;n\in 2\mathbb{Z},
\; |n|>N, \, \alpha = 1,2\}$ is a Riesz basis of projections in
$L^2 ( [0,\pi] , \mathbb{C}^2),$  i.e.,
\begin{equation}
\label{3.3a}  {\bf f} = S_N ({\bf f}) +  \sum_{\alpha=1}^2  \sum_{|n|>N} P_{n,\alpha} ({\bf f})
 \quad    \forall  {\bf f} \in L^2 ( [0,\pi] , \mathbb{C}^2),
\end{equation}
where the series converge unconditionally.
\end{Theorem}

\begin{proof}  In view of Theorem \ref{srl},
there is an $N = N(v, bc)$ such that   the projections
$$
 S_N (\zeta) = \frac{1}{2\pi i} \int_{\partial R_{NT}} (\lambda -
L(\zeta v))^{-1} d\lambda, \quad P_{n,\alpha} (\zeta) =
\frac{1}{2\pi i} \int_{\partial D_n^\alpha} (\lambda - L (\zeta
v))^{-1} d\lambda, $$ $|n|>N, \; \alpha =1,2,$ are well-defined
for $|\zeta| \leq 1$ and depend continuously (even analytically)
on $\zeta.$ Therefore, their dimensions $$ \dim S_N (\zeta) =
trace \,S_N (\zeta), \quad  \dim P_{n,\alpha} (\zeta)  = trace \,
P_{n,\alpha} (\zeta) $$ are constants as continuous integer-valued
functions.  This proves (\ref{3.2}).

Next we prove (\ref{3.3}). For periodic, antiperiodic  and
Dirichlet boundary conditions (\ref{3.3}) was proved in
\cite[Theorem 3]{DM20}; here we follow the same approach.

 For large enough  $N$ the  series in
(\ref{1.11}) converges (see formula (\ref{2.410}) the proof of
Theorem~\ref{srl}); therefore,
\begin{equation}
\label{3.4} P_{n,\alpha} - P_{n,\alpha}^0 = \frac{1}{2\pi i}
\int_{\partial D^\alpha_n} \sum_{s=0}^\infty K_\lambda   (K_\lambda
V K_\lambda )^{s+1} K_\lambda d\lambda.
\end{equation}

Let $\Phi= \{\varphi^1_k,\varphi^2_k, k\in \mathbb{Z}\} $ and $
\tilde{\Phi}=\{\tilde{\varphi}^1_k,\tilde{\varphi}^2_k, k\in
\mathbb{Z}\}$ be the  Riesz basis (consisting of eigenfunctions of
the operator $L_{bc}^0$) and its biorthogonal system that are
constructed in Lemma~\ref{lem3}. We are going to prove (\ref{3.3})
by estimating the  Hilbert--Schmidt norms $\|P_{n,\alpha} -
P_{n,\alpha}^0\|^*_{HS}$ with respect to the basis $\Phi.$

Recall that $$ (\|P_{n,\alpha} - P_{n,\alpha}^0\|^*_{HS})^2
=\sum_{\mu,\nu=1}^2 \sum_{m,k \in 2\mathbb{Z}} |\langle
(P_{n,\alpha} - P_{n,\alpha}^0)\varphi_m^\mu,\tilde{\varphi}_k^\nu
\rangle|^2. $$ By (\ref{3.4}), we obtain $$ \langle (P_{n,\alpha}
- P_{n,\alpha}^0)\varphi_m^\mu,\tilde{\varphi}_k^\nu \rangle =
\sum_{s=0}^\infty I_{n,\alpha}^{\nu,\mu}(s,k,m), $$ where
\begin{equation}
\label{3.5} I_{n,\alpha}^{\nu,\mu}(s,k,m)= \frac{1}{2\pi i}
\int_{\partial D_n^\alpha}\langle K_\lambda(K_\lambda V K_\lambda
)^{s+1}K_\lambda \varphi_m^\mu,\tilde{\varphi}_k^\nu \rangle
d\lambda.
\end{equation}
Therefore,
$$
\sum_{|n|>N} (\|P_{n,\alpha} - P_{n,\alpha}^0\|^*_{HS})^2 \leq
\sum_{s,t=0}^\infty \sum_{|n|>N} \sum_{\mu,\nu=1}^2\sum_{m,k \in
\mathbb{Z}} |I_{n,\alpha}^{\nu,\mu}(s,k,m)|\cdot
|I_{n,\alpha}^{\nu,\mu}(t,k,m)|.
$$
Now, the Cauchy inequality implies
\begin{equation}
\label{3.7} \sum_{|n|>N} (\|P_n - P_n^0\|^*_{HS})^2 \leq
 \sum_{s,t=0}^\infty (A(s))^{1/2}(A(t))^{1/2},
\end{equation}
where
\begin{equation}
\label{3.8} A(s)=\sum_{|n|>N} \sum_{\mu,\nu=1}^2\sum_{m,k \in
\mathbb{Z}} |I_{n,\alpha}^{\nu,\mu}(s,k,m)|^2.
\end{equation}

 Of course, $A(s)$ depends on $\alpha$ and  $N$ but that
dependence is suppressed in the notation.

In view of (\ref{1.16}) and (\ref{3.5}), it follows that
\begin{equation}
\label{3.9} I_{n,\alpha}^{\nu,\mu}(s,k,m)=    \frac{1}{2\pi
i}\int_{\partial D^\alpha_n} \sum_{\gamma_1,..,\gamma_s=1}^2
\sum_{j_1, \ldots j_s}   \frac{w^{\nu
\gamma_1}(k + j_1)}{ (\lambda -\tau_{\nu}-k)}  \times
\end{equation}
$$ \times
 \frac{w^{\gamma_1 \gamma_2}(j_1+j_2) \cdots w^{\gamma_{s-1} \gamma_s}(j_{s-1}+j_s)
w^{\gamma_s \mu}(j_s +m)} {(\lambda -\tau_{\gamma_1}-j_1)
 \cdots (\lambda -\tau_{\gamma_s}-j_s ) (\lambda -\tau_{\mu}-m )} d\lambda.
$$
By the Cauchy formula, if  $ n \not\in \{k,j_1,\ldots,j_s,m\}$
then
\begin{equation}
\label{3.10} \int_{\partial D^\alpha_n}
 \frac{w^{\nu
\gamma_1}(k + j_1) w^{\gamma_1 \gamma_2}(j_1+j_2) \cdots
w^{\gamma_s \mu}(j_s +m)} {(\lambda -\tau_{\nu}-k)
(\lambda -\tau_{\gamma_1}-j_1)
 \cdots (\lambda -\tau_{\gamma_s}-j_s ) (\lambda -\tau_{\mu}-m )}
d\lambda =0.
\end{equation}
This observation is crucial for the proof. We remove from the sum in
(\ref{3.9}) the terms which integrals are zeros and after that
estimate the remaining terms by absolute value as follows.

Let $r$ be the  $\ell^2 (2\mathbb{Z})$--sequence defined in
(\ref{2.18}). We set
\begin{equation}
\label{3.13} B(z, k,j_1,\ldots,j_s,m)= \frac{r(k+j_1)r(j_1 +j_2)
\cdots r(j_{s-1}+j_s) r(j_s +m)}{|z -k||z -j_1| \cdots |z -j_s| |z
-m |}
\end{equation}
for $s>0,$  and
\begin{equation}
\label{3.14} B(z,k,m)= \frac{r(m+k)}{|z -k||z -m |}
\end{equation}
in the case when $s=0 $ and there are no $j$-indices.

\begin{Lemma}
\label{lemsr} In the above notations, we have
\begin{equation}
\label{136} A(s) \leq 4\rho (2C)^s \left( B_1 (s)+B_2 (s)+B_3 (s)+B_4 (s)
\right ),
\end{equation}
with $C=C(\rho)$  and
\begin{equation}
\label{137} B_1 (s)= \sum_{|n|>N} \sup_{|z-n|=\rho} \left (
\sum_{j_1, \ldots, j_s}
 B(z, n,j_1,\ldots,j_s,n) \right )^2;
\end{equation}
\begin{equation}
\label{138} B_2 (s)= \sum_{|n|>N} \sum_{k\neq
n}\sup_{|z-n|=\rho}\left ( \sum_{j_1, \ldots, j_s}  B(z,
k,j_1,\ldots,j_s,n) \right )^2;
\end{equation}
\begin{equation}
\label{139} B_3 (s)= \sum_{|n|>N} \sum_{m\neq n} \sup_{|z-n|=\rho}
\left ( \sum_{j_1, \ldots, j_s}  B(z, n,j_1,\ldots,j_s,m) \right
)^2;
\end{equation}
\begin{equation}
\label{140}  B_4 (s)= \sum_{|n|>N} \sum_{m,k\neq
n}\sup_{|z-n|=\rho}\left (\sum_{j_1, \ldots, j_s}^*  B(z,
k,j_1,\ldots,j_s,m) \right )^2,   \quad s\geq 1,
\end{equation}
where the symbol $*$ over the sum in the parentheses means that at
least one of the indices $j_1, \ldots, j_s $ is equal to $n.$
\end{Lemma}

\begin{proof}
In view of (\ref{3.8}),  we have
$$
A(s) \leq A_1 (s) + A_2 (s) + A_3 (s) + A_4 (s),
$$
where
$$
A_1 (s) = \sum_{|n|>N} \sum_{\nu, \mu=1}^2 \left |
I_{n,\alpha}^{\nu, \mu}  (s,n,n) \right |^2 , \quad
A_2 (s) = \sum_{|n|>N} \sum_{\nu, \mu=1}^2
\sum_{k \neq n} \left | I_{n,\alpha}^{\nu, \mu}  (s,k,n)\right |^2,
$$
$$
A_3 (s) = \sum_{|n|>N} \sum_{\nu, \mu=1}^2 \sum_{m \neq n} \left |
I_{n,\alpha}^{\nu, \mu}  (s,n,m)\right |^2,
\quad
A_4 (s) = \sum_{|n|>N} \sum_{\nu, \mu=1}^2 \sum_{k,m \neq n}
 \left | I_{n,\alpha}^{\nu, \mu}  (s,k,m)\right |^2.
$$
So, the lemma will be proved if we show that $A_i (s)
\leq  4\rho (2C)^s B_i (s),$ $ \; i=1, 2, 3, 4. $

If
$\lambda \in \partial D^\alpha_n $ and $z= \lambda - \tau_\alpha,$
then we have
\begin{equation}
\label{3.16} \left | \frac{w^{\nu
\gamma_1}(k + j_1) w^{\gamma_1 \gamma_2}(j_1+j_2) \cdots w^{\gamma_{s-1}
\gamma_s}(j_{s-1}+j_s)
w^{\gamma_s \mu}(j_s +m)} {(\lambda -\tau_{\nu}-k)(\lambda -
\tau_{\gamma_1}-j_1)
 \cdots (\lambda -\tau_{\gamma_s}-j_s ) (\lambda -\tau_{\mu}-m )}
 \right |
\end{equation}
 $$
 \leq C^s B(z, k,j_1,\ldots,j_s,m),  \quad \text{where}  \;\; C=C(\rho)>1.
$$
In order to prove (\ref{3.16}) it is enough to show
that
\begin{equation}
\label{3.17} |z +\tau_\alpha -\tau_\beta -j|
 \geq \frac{1}{C}|z-j|, \quad  \text{if} \quad |z-n|=\rho, \; \;
 \beta \neq \alpha.
\end{equation}
If $j=n,$ then  by the choice of $\rho$ in (\ref{2.29}) we have $$
|z +\tau_\alpha -\tau_\beta -n| \geq |\tau_1-\tau_2| - |z-n|
=|\tau_1-\tau_2| - \rho\geq \rho =|z-n|.$$

 Otherwise, $|n-j|\geq 2,$ so taking into account that $|Re \,
(\tau_\alpha -\tau_\beta)|\leq 1$  due to (\ref{15}), we obtain $$
|z +\tau_\alpha -\tau_\beta -j|\geq |n-j|- \rho -|Re \,
(\tau_\alpha -\tau_\beta)| \geq |n-j|- \rho -1. $$ Since $|z-j|
\leq  |n-j|+\rho, $ it is enough to find
 a constant $C$ such that $$ |n-j|- \rho -1 \geq
\frac{1}{C}(|n-j|+\rho),  $$ or equivalently, $ (C-1)|n-j| \geq
(C+1)\rho +C. $  For $|n-j|=2 $ the latter inequality is
equivalent to $ C(1-\rho ) \geq 2+\rho.$  Therefore, (\ref{3.17})
holds with $C = C(\rho) = (2+\rho)/(1-\rho).$

Now,   (\ref{3.9}) and (\ref{3.16}) imply that
\begin{equation}
\label{3.22}
\left |I_{n,\alpha}^{\nu, \mu}  (s,k,m)\right | \leq \rho (2C)^s  \sup_{|z-n|=\rho}  \sum_{j_1, \ldots, j_s} B(z,k, j_1, \ldots, j_s, m),
\end{equation}
where $C=C(\rho)$ is the constant from (\ref{3.16}).
Therefore,  in view of (\ref{137}) -- (\ref{139}), we obtain
$$  A_i (s) \leq  4\rho (2C)^s   B_i (s), \quad  i=1,2,3. $$

Finally, taking into account (\ref{3.10}) we remove
from the sum in the right-hand side of (\ref{3.22})
the terms associated with  sets indices $k, j_1, \ldots, j_s, m $
such that $n \not\in \{ k, j_1, \ldots, j_s, m  \}. $   This
 leads to the following improvement of (\ref{3.22}):
$$
\left |I_{n,\alpha}^{\nu, \mu}  (s,k,m)\right | \leq \rho (2C)^s
\sup_{|z-n|=\rho}  \sum^*_{j_1, \ldots, j_s} B(z,k, j_1, \ldots, j_s, m),
\quad k, m \neq n.
$$
In view of (\ref{140}), this yields $A_4 (s)
\leq 4\rho (2C)^s B_4 (s),$   which completes the proof.

\end{proof}

\begin{Proposition}
\label{prop1} In the above notations,
\begin{equation}
\label{148}  B_\nu (s) \leq      C_1 \|r\|^2 a_N^{2s},\quad \nu
=1,2,3, \qquad  B_4 (s) \leq  C_1  s \|r\|^4 a_N^{2(s-1)}, \; s\geq
1,
\end{equation}
where
\begin{equation}
\label{150} a_N = \frac{30}{\rho} \left ( \frac{\|r\|^2}{\sqrt{N}} +
(\mathcal{E}_{N} (r))^2 \right )^{1/2}.
\end{equation}
and $C_1$ is an absolute constant.
\end{Proposition}
If $\rho =1/2, $   then Proposition \ref{prop1}  is identical
with Proposition~{6}  in \cite{DM20}.
Moreover, the proof is one and the same for any $\rho >0 $
but $\rho $ appears in the formula (\ref{150}).
Therefore, we omit the proof of Proposition~\ref{prop1}.

Now we complete the proof of (\ref{3.3}).
 Lemma~ \ref{lemsr} together with the inequalities (\ref{148})
and (\ref{150}) in Proposition \ref{prop1} imply that
\begin{equation}
\label{201} A (s) \leq 16C_1 (2C)^s  \| r\|^2 ( 1 + \|
r\|^2/a_N^2 ) (1+s) a_N^{2s},
\end{equation}
\begin{equation}
\label{202} \left (A (s)A (t) \right
)^{1/2} \leq 16C_1  \| r\|^2 ( 1 + \|
r\|^2/a_N^2 )  (1+s)(1+t)  (2C a_N)^{s+t}.
\end{equation}
By (\ref{150}),  $a_N \to  0 $ as $N \to \infty, $
so $2C a_N < 1$ if $N$ is chosen sufficiently large.
Then,  the inequality (\ref{202})
guarantees that the series on the right-hand side of (\ref{3.7})
converges, which implies that (\ref{3.3}) holds.

Finally,  we apply Theorem \ref{thm2} to the systems of
projections $$\{S_N, \;P_{n,\alpha}, \; |n|>N, \, \alpha = 1,2\},
\quad \{S^0_N, \;P^0_{n,\alpha}, \; |n|>N, \, \alpha = 1,2\}.$$
The existence of  the Riesz basis $\Phi   $ constructed in
Lemma~\ref{lem3} implies that the system $\{S^0_N,
\;P^0_{n,\alpha}, \; |n|>N, \, \alpha = 1,2\}$ is a Riesz basis of
projections in $L^2 ( [0,\pi] , \mathbb{C}^2),$ and by (\ref{3.2})
and (\ref{3.3}) the conditions (\ref{p14}) and (\ref{p15}) are
satisfied. Hence, by Theorem~\ref{thm2}, $\{S_N, \;P_{n,\alpha},
\; |n|>N, \, \alpha = 1,2\}$ is a Riesz basis of projections in
$L^2 ( [0,\pi] , \mathbb{C}^2).$
\end{proof}

Theorem \ref{thm1} immediately implies the following.

\begin{Corollary}
\label{cor1} The spectrum of $L_{bc}$ is discrete. Each of the
discs $D_n^\alpha \,,  \; \alpha =1,2, \; n\in 2\mathbb{Z}, \;
|n|> N,$ contains exactly one simple eigenvalue  of $L_{bc},$  and
the numbers of eigenvalues of $L^0_{bc}$ and $L_{bc}$ (counted
with their algebraic multiplicity) in $R_{NT} $  are equal, namely
\begin{equation}
\label{3.42} \# \left ( Sp\, (L_{bc}) \cap R_{NT}  \right ) =\# \left (
Sp\, (L^0_{bc}) \cap R_{NT}  \right) = 2N.
\end{equation}
\end{Corollary}

In view of Corollary \ref{cor1},  the spectrum of the operator $L_{bc}$
could be described by saying that with exception of finitely many points
it consists of simple eigenvalues  $\lambda_{n, \alpha} $ that are "close"
to the corresponding points in the spectrum of the free operator
$L^0_{bc}$
$$
Sp (L^0_{bc}) =\{ \lambda_{n,\alpha}^0 = n+ \tau_\alpha, \;\; \alpha=1,2;
\;\; n \in 2\mathbb{Z}  \}.
$$
The distance $|\lambda_{n, \alpha} -\lambda_{n, \alpha} ^0 | $
could be estimate by the norms
 \begin{equation}
 \kappa_{n,\alpha} = \|P_{n, \alpha} - P_{n, \alpha}^0\|, \quad
 n \in 2\mathbb{Z}, \;\; |n|>N, \;\; \alpha=1,2,
 \end{equation}
and the terms $w^{\alpha \alpha} (2n) $ from the matrix representation
of the operator of multiplication $V$   (see Lemma~ \ref{lem10}).
This leads to the following statement.

\begin{Theorem}
\label{asr}
In the above notations,
\begin{equation}
\label{3.50}
\sum_{|n|>N} |\lambda_{n, \alpha} - n - \tau_\alpha |^2 <\infty,
\quad \alpha = 1,2.
\end{equation}
\end{Theorem}

\begin{proof}
Let $\Phi=\{\varphi^1_n, \varphi^2_n, \, n \in 2\mathbb{Z}\} $ be the basis of
eigenvectors of $L_{bc}^0 $ constructed in Lemma~\ref{lem3},
and let  $\tilde{\Phi}=\tilde{\varphi}^1_n, \tilde{\varphi}^2_n, \,
n \in 2\mathbb{Z}\} $
be its biorthogonal system.  We have
$$
L_{bc}^0  \varphi^\alpha_n = \lambda_{n, \alpha}^0 \varphi^\alpha_n, \quad
P^0_{n,\alpha} \varphi^\alpha_n = \varphi^\alpha_n, \quad \alpha = 1,2.
$$
and  (since $\tilde{\varphi}^\alpha_n $ are  eigenvectors of the
adjoint operator $(L_{bc}^0)^* $
corresponding to eigenvalues $\overline{\lambda_{n, \alpha}^0}$)
 \begin{equation}
 \label{3.101}
 (L_{bc}^0)^*  \tilde{\varphi}^\alpha_n =
 \overline{\lambda_{n, \alpha}^0} \tilde{\varphi}^\alpha_n , \quad
 \quad \alpha = 1,2.
\end{equation}
We set
$$
\psi_n^\alpha = P_{n, \alpha} \varphi^\alpha_n,
\quad \alpha = 1,2; \;\; n \in 2\mathbb{Z}, \;\; |n|>N.
$$
Then we have
\begin{equation}
 \label{3.102}
L_{bc}  \psi_n^\alpha = \lambda_{n, \alpha} \psi_n^\alpha
\end{equation}
and
\begin{equation}
 \label{3.103}
\| \psi_n^\alpha -\varphi^\alpha_n\| =
 \| (P_{n, \alpha}-P_{n, \alpha}^0   )\varphi^\alpha_n\|
\leq \kappa_{n,\alpha} \|  \varphi^\alpha_n\| \leq C \kappa_{n,\alpha},
\end{equation}
where $C $   is the norm of the isomorphism
$A$  from Lemma~\ref{lem3}, so $C=C(bc).$
By (\ref{3.102}),
$$
\lambda_{n, \alpha} \langle \psi_n^\alpha, \tilde{\varphi}^\alpha_n  \rangle =
\langle L_{bc} \psi_n^\alpha, \tilde{\varphi}^\alpha_n  \rangle
=\langle L^0_{bc} \psi_n^\alpha, \tilde{\varphi}^\alpha_n  \rangle
+\langle V\psi_n^\alpha, \tilde{\varphi}^\alpha_n  \rangle
$$
In view of (\ref{3.101}),
$$
\langle L^0_{bc} \psi_n^\alpha, \tilde{\varphi}^\alpha_n  \rangle
= \langle  \psi_n^\alpha, (L^0_{bc})^*\tilde{\varphi}^\alpha_n  \rangle=
\langle  \psi_n^\alpha, \overline{\lambda_{n, \alpha}^0}\tilde{\varphi}^\alpha_n  \rangle
=\lambda_{n, \alpha}^0
\langle  \psi_n^\alpha, \tilde{\varphi}^\alpha_n  \rangle.
$$
Therefore, we obtain,
$$
(\lambda_{n, \alpha} - \lambda_{n, \alpha}^0 ) \langle
\psi_n^\alpha, \tilde{\varphi}^\alpha_n  \rangle=
\langle V\psi_n^\alpha, \tilde{\varphi}^\alpha_n  \rangle,
$$
which leads to the formula
\begin{equation}
 \label{3.105}
 \lambda_{n, \alpha} - \lambda_{n, \alpha}^0=
 \frac{\langle V\psi_n^\alpha, \tilde{\varphi}^\alpha_n  \rangle}
 {\langle  \psi_n^\alpha, \tilde{\varphi}^\alpha_n  \rangle}.
 \end{equation}
By (\ref{3.103}),  it follows that
$$
\langle  \psi_n^\alpha, \tilde{\varphi}^\alpha_n  \rangle =
\langle   \varphi_n^\alpha, \tilde{\varphi}^\alpha_n  \rangle
+\langle  (\psi_n^\alpha -\varphi_n^\alpha),
\tilde{\varphi}^\alpha_n  \rangle =
1+O (\kappa_{n,\alpha}).
$$
On the other hand,
$$
\langle V\psi_n^\alpha, \tilde{\varphi}^\alpha_n  \rangle=
\langle V\varphi_n^\alpha, \tilde{\varphi}^\alpha_n  \rangle
+\langle V(\psi_n^\alpha-\varphi_n^\alpha), \tilde{\varphi}^\alpha_n  \rangle.
$$
By Lemma \ref{lem10} and (\ref{m3}), we have
$$
\langle V\varphi_n^\alpha, \tilde{\varphi}^\alpha_n  \rangle = w^{\alpha \alpha} (2n), \quad
\text{where} \quad \sum_n    w^{\alpha \alpha} (2n)|^2 < \infty.
$$
In view of (\ref{3.103}),
$$
\langle V(\psi_n^\alpha-\varphi_n^\alpha), \tilde{\varphi}^\alpha_n  \rangle
=\langle(\psi_n^\alpha-\varphi_n^\alpha, V^* \tilde{\varphi}^\alpha_n  \rangle
= O (\|\psi_n^\alpha-\varphi_n^\alpha \|) = O (\kappa_n)
$$
because the functions  $\tilde{\varphi}^\alpha_n$ are uniformly bounded
due to Lemma~\ref{lem3}, Formulas (\ref{21*}) and (\ref{22*}).

Therefore, by (\ref{3.105}), we obtain
$$
\lambda_{n, \alpha} - \lambda_{n, \alpha}^0 = \frac{w^{\alpha \alpha} (2n) + O(\kappa_n)}{1+O(\kappa_n)},
$$
From here (\ref{3.50}) follows because
$\sum_n |w^{\alpha \alpha} (2n)|^2 < \infty  $  by (\ref{m3})  and  $\sum_n \kappa_n^2 < \infty  $
by (\ref{3.3}).
\end{proof}

\section{Bari--Markus property in the case of regular but not
strictly regular boundary conditions}

We use the notations of Section 5.  For regular but not strictly
regular  $bc$  Theorem \ref{rl}  gives  the following
localization of the spectrum of the Dirac operator   $L_{bc}:$
 $$ Sp\, (L_{bc}) \subset R_{NT} \cup \bigcup_{|n|>N}
D_n . $$

Let us consider the Riesz projections associated with $L_{bc}$
\begin{equation}
\label{6.1} S_N = \frac{1}{2\pi i} \int_{\partial R_{NT}} (\lambda
- L)^{-1} d\lambda, \quad P_n = \frac{1}{2\pi i} \int_{\partial
D_n} (\lambda - L)^{-1} d\lambda,
\end{equation}
and let $S^0_N$ and $P_n^0$ be the Riesz projections associated
with the free operator $L^0_{bc}.$

\begin{Theorem}
\label{thm11} Suppose $L_{bc}$ and $L_{bc}^0$ are, respectively,
the Dirac operator with an $L^2$ potential $v$ and the
corresponding free Dirac operator, subject to regular but not
strictly regular boundary conditions $bc.$ Then, there is an $N
\in 2\mathbb{N}$ such that the Riesz projections $S_N, \,P_n$ and
$ S_N^0, \, P_n^0, \; n\in 2\mathbb{Z}, \; |n|>N, $ associated
with $L$ and $L^0$ are well defined by (\ref{3.1}), and we have
\begin{equation}
\label{6.2}  \dim P_n= \dim P_n^0= 2, \quad \dim S_N = \dim
S_N^0=2N;
\end{equation}
\begin{equation}
\label{6.3}   \sum_{|n|>N} \|P_n - P_n^0\|^2 < \infty.
\end{equation}
Moreover,  the system $\{S_N; \;P_n, \; n\in 2\mathbb{Z}, \; |n|>N
\, \}$ is a Riesz basis of projections in $L^2 ( [0,\pi] ,
\mathbb{C}^2),$  i.e.,
\begin{equation}
\label{6.3a}  {\bf f} = S_N ({\bf f}) + \sum_{|n|>N} P_n ({\bf f})
 \quad    \forall  {\bf f} \in L^2 ( [0,\pi] , \mathbb{C}^2),
\end{equation}
where the series converge unconditionally.
\end{Theorem}

\begin{proof}  One may prove the theorem by repeating
(with a few obvious adjustments) the
 proof of Theorem \ref{thm1}.  Therefore, the proof is omitted.
\end{proof}

Theorem \ref{thm11} immediately implies the following.

\begin{Corollary}
\label{cor2} The spectrum of $L_{bc}$ is discrete. Each of the
discs $D_n,   \; n\in 2\mathbb{Z}, \; |n|> N,$ contains exactly
two eigenvalues (counted with algebraic multiplicity) of $L_{bc},$
and the numbers of eigenvalues of $L^0_{bc}$ and $L_{bc}$ (counted
with algebraic multiplicity) in $R_{NT} $ are equal, namely
\begin{equation}
\label{6.4} \# \left ( Sp\, (L_{bc}) \cap R_{NT}  \right ) =\#
\left ( Sp\, (L^0_{bc}) \cap R_{NT}  \right) = 2N.
\end{equation}
\end{Corollary}

\section{Miscellaneous; pointwise convergence and equiconvergence}
 
1. Suppose  that $L^0_{bc} $  is the free Dirac operator considered
with  regular  boundary  conditions $(bc) $
given by the matrix  $\left [
\begin{array}{cccc}
1 & b & a & 0\\ 0 & d& c & 1
\end{array}
\right ]$ in (\ref{8}). Let  $\Phi= \{\varphi^1_k,\varphi^2_k, k\in \mathbb{Z}\}
$  be the corresponding Riesz  basis in $L^2 ([0,\pi], \mathbb{C}^2)
$ consisting of eigenfunctions and associated functions of
$L^0_{bc}, $ which is
 constructed, respectively, in
Lemma~\ref{lem3} if $bc$ is strictly regular, in Lemma~\ref{lem5} if
$bc$ is periodic type, and in Lemma~\ref{lem6} otherwise. Then we
have
\begin{equation}
\label{c0} \sum_{m\in 2\mathbb{Z}}\sum_{\mu=1}^2 \left \langle
\begin{pmatrix}  f\\g  \end{pmatrix},
\tilde{\varphi}_m^\mu  \right \rangle \varphi_m^\mu =
\begin{pmatrix}  f\\g  \end{pmatrix}, \quad \forall f,g
\in L^2 ([0,\pi], \mathbb{C}),
\end{equation}
where the series converges unconditionally in $L^2 ([0,\pi],
\mathbb{C}^2). $ The following statement gives sufficient conditions
for point-wise convergence of the series in (\ref{c0}) and explains
what is its sum for each $x \in [0,\pi].$

{\bf Pointwise Convergence Theorem.} {\em If $f,g: [0,\pi] \to
\mathbb{C} $
   are
functions of bounded variation which are continuous at 0 and $\pi,$
then
\begin{equation}
\label{c2} \lim_{M \to \infty}  \sum_{m=-M}^M  \left [  \left
\langle  \begin{pmatrix}  f\\g  \end{pmatrix}, \tilde{\varphi}_m^1
\right \rangle \varphi_m^1 (x) + \left \langle
\begin{pmatrix}  f\\g  \end{pmatrix}, \tilde{\varphi}_m^2
\right \rangle \varphi_m^2 (x) \right ] =
\begin{pmatrix}  \tilde{f}(x) \\ \tilde{g}(x)     \end{pmatrix}
\end{equation}
where
 \begin{equation}
\label{c03}
 \begin{pmatrix}   \tilde{f} (x) \\ \tilde{g} (x) \end{pmatrix}
=\frac{1}{2} \begin{pmatrix} f(x-0) + f(x+0) \\ g(x-0) + g(x+0)
\end{pmatrix}   \quad   \text{for}  \; \; x \in (0,\pi)
\end{equation}
and
\begin{equation}
\label{c3}
\begin{pmatrix}  \tilde{f}(x) \\ \tilde{g}(x)     \end{pmatrix} =
\begin{cases}
 \frac{1}{2} \begin{pmatrix} f(0) -b f(\pi) -a g(0) \\
   \frac{d}{bc-ad} f(0) + g(0) -\frac{b}{bc-ad} g(\pi)  \end{pmatrix}
 &  \text{if} \;\; x=0,   \\
    \frac{1}{2} \begin{pmatrix}  -\frac{c}{bc-ad}
    f(0)+f(\pi)+ \frac{a}{bc-ad} g(\pi) \\ -d f(\pi)  - c g(0) +g(\pi)
    \end{pmatrix}
 &  \text{if} \;\; x=\pi.
     \end{cases}
\end{equation}
Moreover, if both $f$ and $g$ are  continuous on some closed
subinterval of $(0,\pi)$ then the convergence (\ref{c2}) is uniform
on that interval.} \vspace{3mm}

2. Next, suppose that $v$ is an $L^2 ([0,\pi]) $ Dirac potential and
consider the operator $L_{bc} (v).$

For strictly regular $bc,$ Theorem~\ref{thm1} shows that there is a
Riesz basis of projections; Formula (\ref{3.3a})
 is an analog of
 (\ref{c0}). Moreover, since the projections $P^\alpha_n
$ that appear in (\ref{3.3a}) are one-dimensional while $\dim S_N =
2N, $ in fact Theorem~\ref{thm1} proves the existence of a Riesz
basis  $\varphi_m^\mu, \; m \in 2\mathbb{Z}, \mu \in \{1,2\}, $
consisting of eigenfunctions and at most finitely many associated
functions of the operator $L_{bc} (v).$

For regular but not strictly regular $bc,$ the existence of Riesz
basis of projections is proven in Theorem~\ref{thm11}, see Formula
(\ref{6.3a}). The Riesz projections $P_m $ that appear in
(\ref{6.3a}) are two-dimensional, and in general it is impossible to
"split" the corresponding two-dimensional subspaces into
one-dimensional so that to get a Riesz basis of functions (see in
\cite{DM26} results about existence and nonexistence  of Riesz
basis of functions in the case of periodic or antiperiodic $bc$).

However, in both cases  the spectral decompositions  (\ref{c0})) of
$L^0_{bc} $  and the spectral decompositions of
$L_{bc} (v)$
given, respectively,  by (\ref{3.3a}) for strictly regular $bc$ and
by (\ref{6.3a}) for regular but not strictly regular $bc,$ converge pointwise
to the same limit, or diverge simultaneously, due to the following.

{\bf Equiconvergence Theorem.} {\em  Let $S_N = S_N (v, bc)$ and $S_N^0 (bc) $ be the
projections defined by (\ref{3.1}), and let   $F: [0,\pi] \to \mathbb{C} $
be a function of bounded variation. Then,
for every regular $bc$ and every $L^2([0,\pi])$-potential $v, $
\begin{equation}
\label{c10}
\left \| \left ( S_N - S^0_N \right ) F \right \|_{\infty} 
\to 0 \quad \text{as} \quad N \to \infty
\end{equation}
}

Proofs and generalizations of these results will be presented elsewhere.
We are thankful to  R. Szmytkowski for bringing our attention to
the point-wise convergence problem of spectral decompositions of 1D Dirac operators.
In the case of separated boundary conditions,
our point-wise convergence results confirm the formula suggested by
 R. Szmytkowski (\cite[Formula 3.14]{SZ01}).


\begin{thebibliography}{99}
\bibitem{B} Bari, N. K. Biorthogonal systems and bases in Hilbert space.
(Russian) Moskov. Gos. Univ. U \v cenye Zapiski Matematika 148(4),
(1951). 69--107.

 \bibitem{Bir1}
G. D. Birkhoff,
On the asymptotic character of the solutions of certain linear
differential equations containing a parameter, 
Trans. Amer. Math. Soc. {\bf 9} (1908),
21--231.

\bibitem{Bir2} G. D. Birkhoff,
Boundary value and expansion problems of ordinary linear
differential equations,  Trans. Amer. Math. Soc. {\bf 9} (1908),
373--395.

\bibitem{DM3} P. Djakov and B. Mityagin, Smoothness of Schr\"odinger
operator potential in the case of Gevrey type asymptotics of the
gaps, J. Funct. Anal. {\bf 195} (2002), 89--128.


\bibitem{DM5} P. Djakov and B. Mityagin,
Spectral triangles of Schr\"odinger operators with complex
potentials. Selecta Math. (N.S.) {\bf 9} (2003),  495--528.


\bibitem{DM6} P. Djakov and B. Mityagin,
Instability zones of a periodic 1D Dirac operator and smoothness of
its potential. Comm. Math. Phys. {\bf 259} (2005), 139--183.


\bibitem{DM7} P. Djakov and B. Mityagin,
Spectra of 1-D periodic Dirac operators and smoothness of potentials.
C. R. Math. Acad. Sci. Soc. R. Can. {\bf 25} (2003), 121--125.



\bibitem{DM15} P. Djakov and B. Mityagin,
Instability zones of periodic 1D Schr\"odinger and Dirac operators
(Russian), Uspehi Mat. Nauk {\bf 61} (2006), no 4, 77--182 (English:
Russian Math. Surveys {\bf 61} (2006), no 4, 663--766).


\bibitem{DM20} P. Djakov and B. Mityagin,
Bari--Markus property for Riesz projections of 1D periodic Dirac
operators,   Math. Nachr.  {\bf 283}  (2010),  no. 3, 443--462.


\bibitem{DM26}  P. Djakov and B. Mityagin,
1D Dirac operators with special periodic potentials, arXiv:1007.3234.

\bibitem{Du58}  N. Dunford,  A survey of the theory of spectral operators, Bull.
Amer. Math. Soc. {\bf 64} (1058), 217 -- 274.

\bibitem{DS71}  N. Dunford, J. Schwartz, Linear Operators, Part III, Spectral
Operators, Wiley, New York, 1971.

\bibitem{GK} I. C. Gohberg, M. G. Krein, Introduction to the theory of linear
non-self-adjoint operators, vol. 18 (Translation of Mathematical
Monographs). Providence, Rhode Island, American Mathematical Society
1969.

\bibitem{HO09} S. Hassi and L. Oridoroga,
Theorem of Completeness for a Dirac-Type Operator with Generalized
$\lambda$-Depending Boundary Conditions,
Integral Equat. Oper. Theor. {\bf 64} (2009), 357-379.


\bibitem{Ke64} G. M. Keselman, On the unconditional convergence of eigenfunction
expansions of certain differential operators,
Izv. Vyssh. Uchebn. Zaved. Mat. {\bf 39} (2)
(1964), 82--93 (Russian).


\bibitem{MalOr}
M. M. Malamud and L. L. Oridoroga,  Theorems of the Completeness for
the Systems of Ordinary Differential Equations,  Functional Analysis
and Applications  {\bf 34}, No 3, (2000),   88--90.


\bibitem{M} Markus, A. S. A basis of root vectors of a dissipative operator.
Dokl. Akad. Nauk SSSR 132 524--527 (Russian); translated as Soviet
Math. Dokl. 1 1960 599--602.

\bibitem{Mi62} V. P. Mikhailov, On Riesz bases in $L^2(0, 1),$  Dokl. Akad. Nauk SSSR
{\bf 144} (1962), 981--984 (Russian).


\bibitem{Mit03}
B. Mityagin,  Convergence of expansions in eigenfunctions of the
Dirac operator. (Russian)
 Dokl. Akad. Nauk 393 (2003), no. 4, 456--459.



\bibitem{Mit04}
B. Mityagin,  Spectral expansions of one-dimensional periodic Dirac
operators. Dyn. Partial Differ. Equ. {\bf 1} (2004), 125--191.

\bibitem{RS}  M. Reed and B. Simon,
Methods of modern mathematical physics, vol. I,
Academic Press, New York, 1975.


\bibitem{Sh79}  A. A. Shkalikov, The basis property of eigenfunctions of an
ordinary differential operator,
Uspekhi Mat, Nauk 34 (1979), 235 -- 236 (Russian)

\bibitem{Sh82}  A. A. Shkalikov, On the basisness property of eigenfunctions of ordinary
differential operators with integral boundary conditions, Vestnik
Mosk. Univ., ser. 1, Math. \& Mech., {\bf 6} (1982), 12 -- 21.

\bibitem{Sh83}  A. A. Shkalikov, Boundary value problems for ordinary differential
equations with
a parameter in the boundary conditions, Trudy Sem. I. G. Petrovskogo 9
(1983),
190 -- 229 (Russian); English transl.: J. Sov. Math. 33 (6) (1986), 1311
-- 1342



\bibitem{SZ01} R. Szmytkowski, Discontinuities in Dirac eigenfunction expansions,
J. Math. Physics {\bf 42}, 4606--4617.

 \bibitem{Tam1}  J. D. Tamarkin,
 Sur quelques points de la theorie des equations di?erentielles lineaires
ordinaires et sur la generalisation de la serie de Fourier,  Rend. Circ. Mat. Palermo
(2) {\bf 34} (1912), 345--382.



  \bibitem{Tam2}  J. D. Tamarkin,
 On some general problems of the theory of ordinary linear dif-
ferential operators and on expansion of arbitrary function into serii,  Petrograd.
1917,  308 p.

 \bibitem{Tam3}  J. D. Tamarkin,
Some general problems of the theory of linear di?erential equations and
expansions of an arbitrary functions in series of fundamental functions,
Math. Z. {\bf 27}   (1928), 1--54.



\bibitem{TY01}
I. Trooshin and M. Yamamoto,  Riesz basis of root vectors of a
nonsymmetric system of first-order ordinary differential operators
and application to inverse eigenvalue problems,  Appl. Anal. {\bf
80}, (2001), 19--51.

\bibitem{TY02}
I. Trooshin and M. Yamamoto, Spectral properties and an inverse
eigenvalue problem for nonsymmetric systems of ordinary differential
operators,  J. Inverse Ill--Posed Probl. {\bf 10} No 6, (2002),
643--658.




\end{thebibliography}
\end{document}